\documentclass{amsart}
\usepackage{amsmath,amssymb,latexsym}
\usepackage[mathscr]{euscript}
\usepackage{enumerate}
\usepackage[all]{xy}

\newcommand{\calA}{{\mathcal{A}}}

\newcommand{\calV}{{\mathcal{V}}}
\newcommand{\calC}{\mathcal{C}}
\newcommand{\calK}{\mathcal{K}}
\newcommand{\calL}{\mathcal{L}}

\newcommand{\calO}{\mathcal{O}}
\newcommand{\calQ}{\mathcal{Q}}

\newcommand{\calS}{\mathcal{S}}
\newcommand{\calT}{\mathcal{T}}
\newcommand{\calU}{\mathcal{U}}
\newcommand{\calZ}{\mathcal{Z}}

\newcommand{\sfG}{\mathsf{G}}
\newcommand{\sfH}{\mathsf{H}}

\newcommand{\sfT}{\mathsf{T}}

\newcommand{\R}{\mathbb{R}}

\newcommand{\N}{\mathbb{N}}

\newcommand{\Hom}{\operatorname{Hom}}

\newcommand{\id}{\operatorname{id}}

\newcommand{\chx}{\mathsf{x}}
\newcommand{\Tu}{\mathrm{T}}
\newcommand{\proofthm}[1]{\noindent\textit{Proof of Thm.~\ref{#1}}.}
\newcommand{\proofend}{\mbox{ }\hfill$\Box$\\[2mm]}
\newcommand{\opstar}[1]{\operatorname{St}(#1)}
\newcommand{\clstar}[1]{\overline{\operatorname{St}}(#1)}
\newcommand{\orbproj}{\omega}

\numberwithin{equation}{section}
\theoremstyle{plain}
        \newtheorem{theorem}{Theorem}[section]
        \newtheorem{lemma}[theorem]{Lemma}
        \newtheorem{proposition}[theorem]{Proposition}
        \newtheorem{corollary}[theorem]{Corollary}

\theoremstyle{definition}
        \newtheorem{definition}[theorem]{Definition}
        \newtheorem{remark}[theorem]{Remark}
        \newtheorem{example}[theorem]{Example}

\title{Geometry of orbit spaces of proper Lie groupoids}
\author{M.J.~Pflaum, H. Posthuma,~\textrm{and} X.~Tang}
\begin{document}
\begin{abstract}
In this paper, we study geometric properties of quotient spaces of proper 
Lie groupoids. First, we construct a natural stratification on such spaces 
using an extension of the slice theorem for proper Lie groupoids of 
Weinstein and Zung. Next, we show the existence of an appropriate metric on 
the groupoid which gives the associated Lie algebroid the structure of a 
singular riemannian foliation. With this metric, the orbit space inherits 
a natural length space structure whose properties are studied. Moreover, 
we show that the orbit space of a proper Lie groupoid can be triangulated. 
Finally, we prove a de Rham theorem for the complex of basic differential 
forms on a proper Lie groupoid. 
\end{abstract}
\address{\newline
Markus J. Pflaum, {\tt markus.pflaum@colorado.edu}\newline
         \indent {\rm Department of Mathematics, University of Colorado,
         Boulder, USA}\newline
        Hessel Posthuma, {\tt H.B.Posthuma@uva.nl}\newline
         \indent {\rm Korteweg-de Vries Institute for Mathematics,
        University of Amsterdam,
         The Netherlands} \newline
        Xiang Tang, {\tt xtang@math.wustl.edu}   \newline
         \indent {\rm  Department of Mathematics, Washington University,
         St.~Louis, USA}}
\maketitle
\section{Introduction}
Proper Lie groupoids are natural generalizations of proper Lie group actions on manifolds.
They appear naturally in foliation theory, cf.\  \cite{MolRF,mm}, and also Poisson
geometry cf. \ \cite{w:linear2}. Orbit spaces of proper Lie group actions and
their stratification theory are well  studied in the mathematical
literature (cf.~\cite{pf:orbit} and references therein). However, much less is known
about the quotient spaces - sometimes called ``coarse moduli spaces'' - of general 
proper Lie groupoids. The goal of this paper is to fill this gap. Interestingly, 
some of our results seem to be new even in the case of a proper Lie group action.

We first study the stratification theory of such quotient spaces. In principle, Zung's
linearization theorem \cite{Zun} for proper Lie groupoids gives a
powerful tool to address this question, but it requires the groupoid to be source 
locally trivial and to have orbits of finite type. Under these assumptions Dragulete, 
cf.\ \cite{dr:thesis}, proved that the quotient space is indeed a Whitney stratified space.
However these conditions seem unnatural and impose some restrictions on the applicability of such a 
theorem as there are examples of proper Lie groupoids which do not satisfy the ``finite type" 
condition, cf.\ \cite{w:linear2}. Here, we prove a full generalization and show that the 
quotient space of any proper Lie groupoid carries a canonical Whitney B stratification 
together with a system of smooth control data. 
The main ingredient of our proof
is a careful study of a neighborhood system of an orbit from which we
are able to drop the assumptions of ``finite type" and ``source locally
trivial". Our analysis also shows that even as a stratified space the quotient space of a 
proper Lie groupoid depends only on the Morita equivalence class of the underlying groupoid.

Next, we study the metric properties of such quotient spaces. For this purpose, we introduce the notion of a ``transversally invariant riemannian metric'' on a Lie groupoid. In contrast to the case of a manifold with a
proper Lie group action, the definition of an invariant metric does not generalize straightforwardly to an arbitrary Lie groupoid,
since there is no canonical representation of a Lie groupoid on its tangent bundle. Therefore, a transversally invariant riemannian metric is only required to be invariant with respect to the 
action of the groupoid on the normal bundle to each orbit, which by contrast is canonically defined.

We prove that any proper Lie groupoid admits a transversally invariant riemannian metric, even a complete one. 
Among the main consequences of this  are the following two results:
\begin{itemize}
\item (Proposition \ref{prop-lie-SRF})
 The unit space of the groupoid carries the structure of a
 singular riemannian foliation in the sense of Molino
 \cite{MolRF} with leaves given by the connected
 components of the orbits. When the groupoid is $s$-connected,
  the leaves coincide with the orbits.
\item (Theorem \ref{Thm:SliceThm2})
 For any orbit $\calO$, there exists a ``metric tubular neighborhood'' $T^\varepsilon_\calO$, controlled 
 by a continuous function $\varepsilon: \calO \to \mathbb{R}_{>0}$, over which the groupoid $\sfG$
 is isomorphic to its linearization $(\sfG_{|\calO}\ltimes N\calO)_{|T^\varepsilon_{\calO,N\calO}}$, 
 where $N\calO$ is the normal bundle to $\calO$.
\end{itemize}
The latter theorem is the above mentioned strengthening of the slice theorem of Zung and Weinstein. 
An important point is that the metric tubular neighborhood need not be invariant. However,
combining both results above, we are able to prove that invariant metric tubular neighborhoods 
exist for a compact orbit in an $s$-connected proper Lie groupoid. This implies that in such 
a groupoid, any compact orbit is stable, cf.\ Proposition \ref{prop:inv-nbhd}.  

In the special case that the quotient space of a proper Lie groupoid happens 
to be smooth, one easily shows that a transversally invariant riemannian metric descends to the quotient
turning the canonical projection into a riemannian submersion. In the general
case where the quotient is singular, we use the two results above to define a
natural metric on the quotient space of a proper Lie groupoid, and prove that
the canonical projection is a submetry. This property implies that the 
quotient space is a length space, and even an Alexandrov space if the 
riemannian curvature of the transversally invariant riemannian metric is bounded below. 

Using our previous results on the existence of smooth control data for
quotient spaces of proper Lie groupoids we succeed to show that 
these spaces can be triangulated. 
This result has significant topological consequences. In particular,
it entails that the orbit space of a proper Lie groupoid has
arbitrarily fine good coverings.   

Finally, we consider the cohomology of the quotient space. There exists a 
natural generalization of basic cohomology of foliations to the category 
of Lie groupoids, which is invariant under Morita equivalence. 
Using our previous result on the triangulability of the quotient space
we prove a de Rham theorem for the basic cohomology of a proper Lie 
groupoid and show that it is finite dimensional in the compact case. 
This result can be viewed as an analog of Schwarz's 
\cite{Schwa:thesis} de Rham theorem for a compact orbispace. 

This paper is organized as follows. For the convenience of the reader
and to explain some notation we recall in Section \ref{sec:prelim} 
several basic concepts of the theory of stratified spaces. 
In Section \ref{sec:linearization}, we recall Zung's linearization theorem 
for proper Lie groupoids \cite{Zun} and prove some useful local properties 
of proper Lie groupoids. In particular, we show that every proper Lie groupoid 
has a transversally invariant riemannian complete riemannian metric $\eta$. For a transversally invariant riemannian metric 
on the underlying proper Lie groupoid we prove in Section
\ref{sec:metric-slice} a metric tubular neighborhood theorem 
thus improving the slice theorem for groupoids from \cite{w:linear2}. 
The natural stratification and control structure of the quotient space  
of a proper Lie groupoid is constructed in Section \ref{sec:stratification}. 
After that, we show in Section \ref{sec:metric} that a transversally invariant riemannian metric 
on a proper Lie groupoid induces a length space metric on the quotient space 
such that the quotient map becomes a submetry. 
Using the stratification and control structure from 
Section \ref{sec:stratification} we derive in Section \ref{sec:triangularization}
that the quotient space even carries a triangulation. 
As an application, we discuss in the final Section \ref{sec:derham}
the basic cohomology of a proper Lie groupoid and prove a de Rham theorem.\\

\noindent{\bf Acknowledgments:} We would like to thank Marius Crainic,
Rui Fernandes, Matias del Hoyo, and Eugene Lerman for helpful discussions. 
In particular Marius Crainic made extensive comments on the first version 
of the paper, and shared with us a preliminary version of \cite{cs}. 
Subsequent discussions with him on the stratification theory of orbit spaces
helped to improve the paper. Matias del Hoyo kindly pointed to us a mistake in the original proof 
of Proposition \ref{Prop:InvMet}. He  and Rui Fernandes suggested a key idea to correct it (See Remark \ref{rmk:rui}). We also thank the referee for constructive 
advice.  Finally, X.T.~and M.J.P.~gratefully acknowledge support by the NSF.
Tang's research has been partially supported by NSF grants DMS 0703775 and DMS 0900985,
Pflaum's research by NSF grants DMS 1066222 and DMS 1105670.
\section{Stratified spaces and smooth structures}
\label{sec:prelim}
We briefly recall the definitions of a stratified space by Mather
\cite{ma:stratification}. Let $X$ be a paracompact Hausdorff space
with countable topology.

\begin{definition}\label{dfn:decomposition}A decomposition of $X$
consists of a locally finite partition $\calZ$ of $X$  into locally
closed subspaces (called pieces) $S\subset X$ satisfying the
following conditions:
\begin{enumerate}[{\rm (DC1)}]
\item Every piece $S$ of $\calZ$ is a smooth manifold with the
induced topology,
\item If $R\cap \overline{S}\ne \varnothing$ for a pair of $R,S\in
\calZ$, then $R$ is contained in $\overline{S}$. $R$ is called
incident to $S$ and this is written as $R\leq S$.
\end{enumerate}
\end{definition}

In this definition, we allow manifolds to have connected components of 
different dimensions. Although this could be avoided in the rest of the paper 
by refining partitions in each step by taking connected components, we have 
chosen this convention because it simplifies the notation.

A continuous map $f:X\to Y$ between decomposed spaces is called a
{\em morphism} of decomposed spaces if for every piece $S$ of
$\calZ_X$ there is a piece $R_S\in \calZ_Y$ such that the image of
$S$ under $f$ is contained in $R_S$, and the restriction map
$f|_S:S\to R_S$ is smooth. For two decompositions $\calZ_1,
\calZ_2$ of a topological $X$, $\calZ_1$ is said to be coarser than
$\calZ_2$ if the identity map is a morphism of decomposed spaces from
$(X, \calZ_2)$ to $(X, \calZ_1)$.

Let $x$ be a point of $X$. Two subsets $A$ and $B$ of $X$ are taken
to be equivalent at $x\in A\cap B$ if there is an open neighborhood $U$ of $x$
with $A\cap U=B\cap U$. The equivalent class of the all subsets that
are equivalent the subset $A\subset X$ is called the germ of $A$ at
$x$, and denoted by $[A]_x$.

\begin{definition}\label{dfn:stratified}
  A stratification of a locally compact
  topological Hausdorff space $X$ is a mapping $\calS$ which assigns to every
  point $x\in X$ the germ $\calS_x$ of some closed subset of $X$ which is locally  
  induced by a decomposition. This means that for every $x$ in $X$, there is a 
  neighborhood $U$ of $x$ and decomposition $\calZ|_U$ of $U$ such that for
   every point $y\in U$, the germ $\calS_y$ agrees with the germ of the piece
  $S|_U\in \calZ|_U$ that contains $y$. A topological space $X$ together
  with a stratification $\calS$ is called a stratified space.
\end{definition}

Every decomposition $\calZ$ of a topological space $X$ naturally
defines a stratification by assigning to every point $x$ the germ
$\calS_x$ of the piece in $\calZ$ that contains $x$. One can show
that every stratified space $(X, \calS)$ has a unique decomposition
$\calZ_\calS$ such that for any open subset $U$ of $X$ and any
decomposition $\calZ$ over $U$ the restriction of $\calZ_\calS$ is
coarser than any $\calZ$. This decomposition is called the {\em
canonical decomposition} of $(X, \calS)$, and a piece of $\calS$ is
called a {\em stratum} of $X$.

Let $(X, \calS)$ be a stratified space and $\calZ_\calS$ the family of its
strata. A singular chart of class $\calC^m$ with $m\in \mathbb{N}\cup \{\infty\}$
is a homeomorphism $\varphi:U\to \varphi(U) \subset \R^n$ from an open set $U$ of $X$
to a locally closed subspace of $\mathbb{R}^n$ such that for every stratum
$S\in \calZ_\calS$ the image $\varphi(U\cap S)$ is a submanifold of
$\mathbb{R}^n$ and the restriction $\varphi_{U\cap S}:U\cap S\to \varphi(U\cap S)$
is a diffeomorphism of class $\calC^m$.

Two singular charts $\varphi_1:U_1\to \varphi_1(U_1) \subset \R^{n_1}$,
$\varphi_2:U_2\to \varphi_2(U_2) \subset \R^{n_1}$ of a stratified space are
called {\em compatible}, if for every $x\in U_1\cap U_2$ there exist $n\geq n_1,n_2$,
an open neighborhood $U_x$ of $x$ with $U_x \subset U_1 \cap U_2$, open neighborhoods
$O_1\subset \mathbb{R}^n$ and
$O_2\subset \mathbb{R}^n$ of $\varphi_1(U_x) \subset \R^{n_1} \subset \R^n$ and
$\varphi_2(U_x)\subset \R^{n_2} \subset \R^n$, together with a
diffeomorphism $h:O_1\to O_2$ of class $\calC^m$ such that
${\varphi_2}_{|U_x}=h\circ {\varphi_1}_{|U_x}$. 
Hereby, we regard $\R^{n_i}$, $i=1,2$ to be embedded
into $\R^n$ via the first $n_i$ coordinates.

Analogously to the standard definition of an atlas for manifold, one can
introduce the notion of a singular atlas for a stratified space.
\begin{definition}\label{dfn:smooth}
  A set of pairwise compatible singular charts of class $\calC^m$ for a stratified
  space $(X,\calS)$ is called a $\calC^m$-atlas for $X$, if the domains of the singular
  charts cover $X$.  A maximal atlas of  singular  charts of class $\calC^m$ for
  $(X,\calS)$ is called a $\calC^m$-structure of $X$. When $m=\infty$, we call it a
  smooth structure on $X$.
\end{definition}
\begin{example}\label{ex:action}
  Let $G$ be a Lie group acting properly on a smooth manifold $M$. For
  $x\in M$, $G_x$ is the isotropy group of $G$ action at $x$. It is
  easy to see that if $x$ and $y$ are in the same orbit of $G$ action,
  then $G_x$ and $G_y$ are conjugate to each other. Define the
  {\em orbit type} of an orbit of the $G$ action to be the conjugacy class
  of $G_x$ for $x$ in the orbit. For a closed subgroup $H< G$,
  $M_{(H)}$ is the subset of $M$ consisting of all $x$ such that $G_x$
  is conjugate to $H$. It is not difficult to show
  \cite[4.3.7, 4.3.11]{pf:book} that the map $x\mapsto M_{(G_x)}$ defines a
  stratification of $M$ by orbit types, and the map
  $G\cdot x\mapsto M_{(G_x)}/G$ defines a stratification of $M/G$. In fact, 
  the connected components of the orbit types $M_{(H)}$ and $M_{(H)}/G$
  define canonical decompositions of $M$ and $M/G$ inducing this stratification, 
  cf.\ Section \ref{pga} 
\end{example}

Given a stratified space $(X,\calS)$ with a smooth structure, one can construct a
structure sheaf $\calC^\infty_X$  of smooth functions as follows. Let $\calU\subset X$
be open, and define $\calC^\infty_X (U) $ as the space of all continuous maps
$f : U \rightarrow \R$ which have the property that for every singular chart
$\chx: V \rightarrow \R^n$ with $U\cap V \neq \emptyset $ there exists a smooth map
$\tilde f : \tilde V \rightarrow \R$ defined on some open set $\tilde V \subset \R^n$
such that $\chx (U\cap V) \subset \tilde V$ and $f_{|U\cap V} = \tilde f \circ \chx$.
Obviously, the spaces $\calC^\infty_X (U) $ form the section spaces of a sheaf
$\calC^\infty_X$ on $X$. The pair $(X,\calC^\infty_X)$ then becomes a so-called
(\emph{reduced}) \emph{differentiable space} (see \cite{GonSalDS}),
in particular Sec.~3.2, for the definition and more information on 
differentiable spaces).
Moreover, the components of a singular chart on $X$ are
elements of some sectional space $\calC^\infty_X (U) $, hence the smooth structure can be
recovered from the structure sheaf $\calC^\infty_X$. That is why we often denote a
smooth structure on a stratified space just by its structure sheaf.  
Let us also note that a sheaf $\calA$ of continuous functions  on $X$  is the sheaf 
of smooth functions of a smooth structure on $X$ if and only if 
the following holds true:
\begin{itemize}
\item $(X,\calA)$ is a differentiable space,
\item  the restriction of 
$\calA$ to each stratum recovers the sheaf of smooth functions on that manifold.
\end{itemize}

For a stratified space $X \subset \R^n$, Mather has introduced the notion of
control data in \cite{Mat:NTS}. Here we need the extension of this concept to
stratified spaces with a smooth structure as provided in \cite[Sec.~3.6]{pf:book}.
To this end let us first recall the notion of a tubular neighborhood
in the sense of \cite{Mat:NTS}.
\begin{definition}
\label{DefTuUmg}
  Let $M$ be a smooth manifold. By a {\it tubular neighborhood} of an embedded
   submanifold
  $S \subset M$ one understands a triple $\sfT =
  (E,\varepsilon,\varphi)$, where $\pi^E : E\rightarrow S$ is a
  smooth vector bundle over $S$ with scalar product $\eta$,
  $\varepsilon: S\rightarrow \R^{>0}$ a continuous map, and $\varphi$ a
  diffeomorphism from $\Tu_{S, E}^\varepsilon :=\big\{ v \in E \mid
  \|v\|^2  = \eta(v,v) < \varepsilon (\pi^E(v))\big\}$ onto an open neighborhood
  $\Tu_S$ of $S$, the so-called {\it total space} of $\sfT$, such that the
  diagram
  \begin{equation}
  \xymatrix{
    \Tu_{S, E}^\varepsilon  \ar[rd]^-{\varphi} \\
    S \ar[u] \ar[r] & M
  }
  \end{equation}
  commutes. The map $\pi_S := \pi^E \circ \varphi^{-1} : \Tu_S \rightarrow S $ is
  called the {\it projection} of the tubular neighborhood, the function
  $\varrho : \Tu_S \rightarrow \R^{>0}$, $p \mapsto  \| \varphi^{-1} (p) \|^2$ its
  {\it tubular function}.

  If $f: M \rightarrow N$ is smooth mapping, then one calls
  a tubular neighborhood $\sfT$ of $S \subset M $
  {\it compatible} with $f$, if $ f \circ \pi_S = f_{|\Tu}$.
\end{definition}

  In case $(M,\eta)$ is a riemannian manifold, and $S\subset M$ a submanifold,
  there is a continuous $\varepsilon : S \rightarrow \R_{>0}$ such that the
  metric exponential map defines a tubular neighborhood $(NS,\varepsilon,\varphi)$
  of $S$ in $M$. Hereby, $\pi^{NS} :NS \rightarrow S$ is the normal bundle of $S$
  in $M$, and $\varphi : =
  \exp_{|\Tu^\varepsilon_{S,NS}}:\Tu^\varepsilon_{S,NS} \rightarrow \Tu_S :=
  \exp \big(\Tu^\varepsilon_{S,NS}\big)$. We will then call $(NS,\varepsilon,\varphi)$
  or for short $\Tu_S$ a metric tubular neighborhood of $S$ in $M$.
  \vspace{2mm}

  The following result guarantees the existence of tubular neighborhoods 
  and is a variant of \cite[Prop.~6.2]{Mat:NTS} and \cite[Thm.~5.1]{w:linear2}.
\begin{theorem}
\label{thm:mettub}
  Let $M$, $Q$, and $R$ smooth manifolds,
  $f: M \rightarrow Q $ a submersion, and $q\in Q$ a point in the
  image of $f$. Let $S := f^{-1} (q)$ and assume that
  $h: M \rightarrow R$ is a smooth map such that $h_{|S}$ is a
  submersion. Then there is a continuous
  $\varepsilon : S \rightarrow \R_{>0}$ and a tubular neighborhood
  $(E,\varepsilon,\varphi)$ of $S$ in $M$ which is compatible with $h$,
  that means the compatibility relation
  \begin{equation}
   \label{eq:comp}
    h \pi_S (p) = h (p)
  \end{equation}
  holds true for all $p \in \Tu^\varepsilon_S$, where $\pi_S :\Tu^\varepsilon_S
  \rightarrow S$ is the projection of the tubular neighborhood.
  Moreover, after possibly shrinking $\varepsilon$ one can even achieve that
  $(\pi_S ,f_{|\Tu^\varepsilon_S}) : \Tu^\varepsilon_S \rightarrow S \times
  f(\Tu^\varepsilon_S)$ is an open embedding.

  In case $G$ is a compact Lie group which acts on $M$ and $Q$ such that
  $f$ is $G$-equivariant and $h$ is $G$-invariant, then
  $\varepsilon$ can be chosen to be $G$-invariant. Moreover, $\varphi$ and
  $\pi_S$ are then equivariant with respect to the natural $G$-action on $E$.
\end{theorem}
\begin{proof}
  The existence of the tubular neighborhood $(E,\varepsilon,\varphi)$
  satisfying the compatibility condition follows
  from \cite{Mat:NTS} (see also \cite[Sec.~3.5]{pf:book}).
  By dimension counting, one checks that
  $\big( T\pi_S , Tf \big)_{|T_{|S}M} : T_{|S}M \rightarrow TS \times T_qQ$
  is an isomorphism of vector bundles, hence after shrinking $\varepsilon$
  appropriately,  $(\pi_S ,f_{|\Tu^\varepsilon_S}) : \Tu^\varepsilon_S
  \rightarrow S \times f(\Tu^\varepsilon_S)$ is an open embedding indeed.

  In the $G$-equivariant case the existence of equivariant tubular
  neighborhoods with the required properties follows from equivariant versions
  of the existence proof for tubular neighborhoods from \cite{Mat:NTS} or
  \cite[Sec.~3.5]{pf:book} or directly from \cite{Kan:}.
\end{proof}

  Now we have the means to define the notions of a tube for a stratum and of
  control data for a stratified space.
\begin{definition}
  Let $(X,\calS)$ be a stratified space, and $S \in \calS$
  one of its strata. By a \textit{tube} of $S$ in $X$ one understands a
  triple $\sfT_S = (\Tu_S, \pi_S, \varrho_S)$ satisfying the following conditions:
\begin{enumerate}[{\rm (TB1)}]
\item
  $\Tu_S$ is an open neighborhood of $S$ in $X$ such that for every
  other stratum $ R \in \calS$ the relation
  $\Tu_{S,R} := \Tu_S \cap R \neq \emptyset$ implies $R \geq S$.
\item
  $\pi_S : \Tu_S \rightarrow S$ is a continuous retraction
  of $S$ such that for every stratum $R > S $ the restriction
  $\pi_{S,R}:= {\pi_S}_{|\Tu_{S,R}} : \Tu_{S,R}\rightarrow S$
  is smooth.
\item
  $\varrho_S : \Tu_S \rightarrow \R^{\geq 0}$ is a continuous
  mapping such that $\varrho_S^{-1} (0) = S$ holds true and such that
  for every stratum $R >S $ the restriction
  $\varrho_{S,R} := {\varrho_S}_{|\Tu_{S,R}} :
  \Tu_{S,R} \rightarrow \R^{\geq 0} $
  is smooth.
\item
  The mapping
  $(\pi_{S,R},\varrho_{S,R}) : \Tu_{S,R} \rightarrow
  S \times \R^{> 0}$
  is a submersion for every pair of strata $R>S$.
\end{enumerate}
  If the stratified space $( X,\calS)$ carries a smooth structure $\calC^\infty$,
  then one calls a tube $\sfT_S = (\Tu_S,\pi_S,\varrho_S)$ of $S$ as being of
  \textit{class} $\calC^\infty$ or \textit{smooth},  if the following axiom is satisfied:
\begin{enumerate}[{\rm (TB1)}]
\setcounter{enumi}{4}
\item
  There exists a covering of $S$ by singular charts
  $\chx : U \rightarrow \R^n$ of $X$ such that
  $\pi_S$ and $\varrho_S$ are induced by tubular neighborhoods
  $\sfT^\chx=(E^\chx,\varepsilon^\chx,\varphi^\chx)$ of $\chx (S\cap U) \subset
  \R^n$ in the sense that
  \begin{equation}
  \begin{split}
  \nonumber
      \pi^\chx  \circ \chx =  \chx \circ {\pi_S}_{|U}, \quad \text{and} \quad
      \varrho^\chx  \circ \chx  = {\varrho_S}_{|U},
  \end{split}
  \end{equation}
  where $\pi^\chx$ is the projection and $\varrho^\chx$ the tubular function
  of $\sfT^\chx$. Note that in this case the functions $\pi_S$ and $\varrho_S$ are
  $\calC^\infty$-mappings.
\end{enumerate}
\end{definition}

\begin{definition}
\label{DefKoDat}
 Given a stratified space $(X,\calS)$, {\it control data} for $(X,\calS)$ consist of a
 family $\left( \sfT_S \right)_{S \in \calS}$ of tubes
 $\sfT_S = (\Tu_S, \pi_S, \varrho_S)$ such that for every pair
 of strata  $R > S$ and all  $ x \in \Tu_S \cap \Tu_R $ with
 $\pi_R (x) \in \Tu_S$ the  control conditions
\begin{enumerate}[{\rm (CD1)}]
\item
\label{KBPiPi}
  $  \pi_S \circ \pi_R (x)  = \pi_S (x)$,
\item
\label{KBRhoPi}
  $  \varrho_S \circ \pi_R (x)  = \varrho_S (x)$,
\end{enumerate}
are satisfied. A stratified space on which some control data exist
is called a {\it controllable} space. If $(X,\calS)$ carries additionally a smooth
structure and if all tubes of the control data are smooth, then $(\sfT_S)_{S\in \calS}$
are called {\it smooth control data}.
\end{definition}
The fundamental existence result for control data which we will later need in this
article goes back to \cite{Mat:NTS}. It essentially states that existence of
control data is guaranteed under the condition that a local condition, namely
Whitney's condition B is satisfied (see \cite[\S 1 \& \S 2]{Mat:NTS} or
\cite[Sec.~4]{pf:book} for details on the Whitney conditions).
\begin{theorem}[cf.~{\cite[Prop.~7.1]{Mat:NTS} and \cite[Thm.~3.6.9]{pf:book}}]
\label{SaExKoDaWhi}
  Let $(X,\calS)$ be a stratified space with a smooth structure and assume that
  it satisfies Whitney's condition B in every singular chart. Then $(X,\calS)$
  possesses a system of smooth control data $(\sfT_S)_{S\in\calS}$.
  Moreover, if $f: X\rightarrow N$ is a smooth stratified submersion to a smooth
  manifold  $N$ then the control data $(\sfT_S)_{S\in\calS}$
  can be chosen to be compatible with $f$.
\end{theorem}
\begin{remark}
\label{Rem:LocCon}
 Note that a stratified space with a smooth structure which fulfills Whitney's 
 condition B is locally contractible (see \cite{pf:orbit}). This fact is later 
 needed for the proof of a de Rham theorem for basic cohomology over orbit spaces of 
 proper Lie groupoids. 
\end{remark}
\section{Proper Lie groupoids}
\label{sec:linearization}
Assume to be given a proper Lie groupoid $\sfG$. Sometimes, we will denote $\sfG$
by $s,t : \sfG_1 \rightrightarrows \sfG_0$ with $s$ and $t$ being its source and
target map, respectively. The unit map of a groupoid $\sfG$ will be 
denoted by $u: \sfG_0 \rightarrow \sfG_1$. For any subset $B\subset \sfG_0$, we write
$\sfG_{|B}$ for the restriction
$(\sfG_{|B})_1 := s^{-1} (B) \cap t^{-1} (B)  \rightrightarrows B$
with the induced source and target maps. If $B$ is a closed submanifold of
$\sfG_0$, then $\sfG_{|B}$ is again a proper Lie groupoid. For each point
$p \in \sfG_0$ denote by
\[
  \calO_p := \{ t(g) \in \sfG_0 \mid  g \in  \sfG_1 \text{ and } p = s(g)\}
\]
the orbit through $p$. Obviously, the orbits give rise to a partition of the
object space $\sfG_0$. Consider the quotient space $X= \sfG_0/ \sfG$
of all orbits with the quotient topology and the natural projection
$\orbproj :\sfG_0\rightarrow X$. According to \cite[Prop.~2.2]{Zun},
$X$ is a Hausdorff topological space, and each orbit $\calO_p$ is a closed
submanifold of $\sfG_0$.  In this article, we will in particular construct a
natural stratification of $X$. As a preparation we first recall the following result.
\begin{proposition}[{cf.~\cite[Prop.~2.2]{Zun}}]
  Given a proper Lie groupoid $\sfG$ the following properties hold true:
  \begin{enumerate}
  \item
    If $\sfH$ is a Hausdorff Lie groupoid which is Morita equivalent to $\sfG$,
    then $\sfH$ is also proper.
  \item
    If $N$ is a submanifold of $\sfG_0$ which intersects an orbit $\calO_p$ at a
    point $p\in \sfG_0$ transversally, and $Z$ is a sufficiently small open
    neighborhood of $p$ in $N$, then the restriction
    $\sfG_{|Z}$ is a proper Lie groupoid which has $p$ as a fixed point.
  \end{enumerate}
\end{proposition}
Consider now a point $p\in \sfG_0$, and let
$\sfG_p := \{ g \in \sfG_1 \mid s(g) = t(g) = p\}$ be the isotropy group at $p$.
Then, $\sfG_p$ is a compact Lie group which acts on the normal space
$N_p \calO_p:= T_p\sfG_0 / T_p \calO_p$. More generally, if $\calO$ denotes an orbit of
$\sfG$ in $\sfG_0$,  the restriction $\sfG_{|\calO}$ is a proper Lie groupoid which
acts on the normal bundle $N\calO := T_{|\calO}\sfG_0 / T\calO$. Following
\cite{Zun}, let us show the latter which obviously also proves that $\sfG_p$ acts
on $N_p\calO$. So let
$g\in \sfG_{|\calO}$, $s(g)=p$, and $n \in N_p\calO$ a normal vector over $p$.
Then $n$ can be represented by a smooth curve  $\delta :(-\varepsilon,\varepsilon) \rightarrow \sfG_0$.
In other words, this means that $\delta (0)=p$ and
$n=\left. \frac{\partial}{\partial \tau} \delta (\tau) \right|_{\tau =0} + T_p \calO \in N_p\calO$.
Next choose a smooth curve $\gamma : (-\varepsilon,\varepsilon) \rightarrow \sfG_1$ such that
$\gamma (0) = g$ and $ s \circ \gamma  = \delta$. Put
\begin{equation}
\label{rep-normal-bundle}
  g \cdot n := \left. \frac{\partial}{\partial \tau} t \big(\gamma (\tau)\big) \right|_{\tau =0}
  + T_{t(g)} \calO ,
\end{equation}
and check that the thus defined normal vector $g \cdot n\in N_{t(g)}\calO  $ does not depend on the
particular choices of $\delta$ and $\gamma$, and that this defines an action of $\sfG_{|\calO}$
on $ N\calO$ indeed.

Before proceeding, let us briefly recall, cf.~\cite[\S 5.4]{mm}, that a morphism $f:\sfG\to \sfH$ of Lie groupoids is 
said to be a \textit{strong equivalence} or an \textit{isomorphism} if it admits an inverse.
It is said to be a \textit{weak equivalence} if the following two conditions are satisfied:
\begin{itemize}
\item[(ES)] the map $t\circ \operatorname{pr}_1:\sfH_1\times_{\sfH_0}\sfG_0\to\sfH_0$ is a surjective submersion, where
the fiber product is taken over the source map of $\sfH$ and $f$,
\item[(FF)] $\sfG_1$ is a  fiber product  as in the following square:
\[
  \xymatrix{\sfG_1\ar[r]^f\ar[d]_{(s,t)}&\sfH_1\ar[d]^{(s,t)}\\
  \sfG_0\times\sfG_0\ar[r]^{(f,f)}&\sfH_0\times\sfH_0}
\]
\end{itemize}
A weak equivalence $f$ induces a homeomorphism of the
quotient space $\sfG_0\slash\sfG_1\cong\sfH_0\slash\sfH_1$.
\vspace{2mm}

The following result  will be useful for extending the slice theorems for proper Lie 
groupoids by Zung  to not necessarily source-locally trivial proper Lie groupoids. 
\begin{proposition}
\label{Prop:sloctriv}
  Let $\sfG$ be a proper Lie groupoid, $p\in \sfG_0$ a point, and $\calO$ the orbit 
  through $p$. Then  there is an open neighborhood $U$ of $p$ such that the 
  restricted groupoid $\sfG_{|U}$ is source-locally trivial and such that the orbit in 
  $\sfG_{|U}$ through $p$, i.e.~the restricted orbit $\calO_{|U} = \calO \cap U$,
  is of finite type. 
\end{proposition}
\begin{proof}
  We consider two cases:
  \begin{enumerate}
  \item The orbit $\calO$ is discrete.
   \item The orbit $\calO$ has positive dimension.
  \end{enumerate}
  In the first case, there exists an open neighborhood $W \subset \sfG_0$
  of $p$ such that $\{ p \} = W\cap \calO$. This implies that in
  the restricted groupoid $\sfG_{|W}$ the point $p$ is a fixed point. 
  By \cite[Prop.~7.3.8]{DufZunPSNF}  there then exists an open neighborhood 
  $U\subset W$ of $p$ such that the restricted groupoid   $\sfG_{|U}$ 
  is source-locally trivial. Since $p$ is a fixed point of $\sfG_{|U}$,
  the restricted orbit $\calO_{|U}$ consists only of one point, hence
  is of finite type. This finishes the first case. 

  So let us consider the second case and assume that  $\dim \calO >0$.
  This implies that each orbit through a sufficiently small neighborhood of $p$ 
  has positive dimension as well. Let us briefly check this. Assume that 
  $p_n \in \sfG_0$ is a sequence of points converging to $p$ such that for every 
  $n\in \N$ the orbit through $p_n$ is discrete. Then $\sfG_{p_n} $ is open and 
  closed in $s^{-1} (p_n)$ for every $n$. Now let $V$ be an open connected neighborhood 
  of $u(p)$ in $\sfG_1$ such that $V \cap s^{-1} (q)$ is connected for every $q$ in a 
  neighborhood of $p$. Observe that $V \cap s^{-1} (p_n) = V \cap \sfG_{p_n} $ for every 
  $n$. Let $g\in V \cap s^{-1} (p)$. Then there exists a sequence
  $g_n \in V \cap \sfG_{p_n}$ converging to $g$. Moreover,
  \[
    s(g) = \lim_{n\rightarrow \infty} s(g_n) = \lim_{n\rightarrow \infty} t(g_n)  = t(g),
  \]
  hence $g \in \sfG_p$. By dimension reasons this implies that the orbit $\calO$ has
  to be discrete which contradicts the assumption on $\calO$. 
  Hence there exists an open neighborhood $W$ of $p$ in $\sfG_0$ such that for every $q\in W$ 
  the orbit $\calO_q$  through $q$ has positive dimension. 
  After possibly shrinking $W$ we can assume that there is a smooth chart 
  $\chx: W \rightarrow B \subset \R^{\dim \sfG_0}$ to an open ball $B$ 
  around the origin such that  $\chx (p)=0$.
  By proper choice of $W$ and the chart $\chx$ one can even achieve that 
  $\chx^{-1} (\R^{\dim \calO} \cap B) = W \cap \calO$, where $\R^{\dim \calO}$ is 
  embedded into $\R^{\dim \sfG_0}$  via the first $\dim \calO$ coordinates.
  Put $U := \chx^{-1} (\frac 12 B )$. The closure $\overline{U}$ then is a manifold with 
  boundary and is diffeomorphic to a closed ball via the chart $\chx$. 
  The restricted target map $t_{|s^{-1} (U)} : s^{-1} (U) \rightarrow \sfG_0$ is a 
  submersion, hence the preimage of  $\overline{U}$ under $t_{|s^{-1} (U)}$ is a 
  submanifold  of $s^{-1} (U)$ with boundary. 
  But this preimage coincides with $N:=t^{-1} (\overline{U}) \cap s^{-1} (U)$. Moreover, 
  the restriction $s_{|N}$ is a submersion, and each of its fibers 
  $t^{-1} (\overline{U}) \cap s^{-1} (q)$,
  $q\in U$ is a compact manifold with boundary by properness of $\sfG$, and since 
  $t_{|s^{-1} (q)} :  s^{-1} (q) \rightarrow \calO_q$ is a submersion.
  
  Let us show that  $s_{|N} : N \rightarrow U $ is a proper map. To this end choose
  some compact subset $K\subset U$, and check that 
  $ s_{|N}^{-1} (K) = t^{-1} (\overline{U}) \cap s^{-1} (K)$. Since 
  $(s,t):\sfG_1 \rightarrow \sfG_0\times \sfG_0$ is a proper map 
  by assumptions on $\sfG$ 
  and since $\overline{U} \times K$ is a compact subset of $\sfG_0\times \sfG_0$,
  one concludes that  $ s_{|N}^{-1} (K)$ is compact, hence $ s_{|N}$ is proper 
  indeed. 

  Finally,  $s_{|N} : N \rightarrow U $ is surjective, since 
  $u (U) \subset t^{-1} (\overline{U}) \cap s^{-1} (U)$.
  Hence   $s_{|N}: N \rightarrow U$ is a surjective 
  proper submersion from a  manifold with boundary to a manifold without boundary. 
  By Ehresmann's Theorem for manifolds with boundary \ref{ThmEhresmannMfdBdry},
  $s_{|N}$ has to be locally trivial with typical fiber given by the manifold with 
  boundary $t^{-1} (\overline{U}) \cap s^{-1} (p)$, if we can yet show that 
  the restriction $s_{|\partial N}$ is a submersion, too. 
  Let us prove that this is the case. To this end it 
  suffices to show that $T_g \partial N + \ker T_gs = T_g\sfG_1$ for all 
  $g \in \partial N$.  
  Choose a (local) bisection $\sigma : V \rightarrow \sfG_1$ around an open neighborhood $V$ 
  of $t(g)$ such that $\sigma (t(g)) =g^{-1}$. This means that $s\circ \sigma = \id_{V}$ 
  and that $t\circ \sigma$ is a diffeomorphism from $V$ onto  its image. 
  Let $\overline{\sigma}$ be the map $V \ni q \mapsto (\sigma(q))^{-1} \in \sfG_1$. 
  Then $\overline{\sigma} (t(g))=g$, and $T_{t(g)}\overline{\sigma} (T_{t(g)} \partial \overline{U})$  
  is a subspace of $T_g \partial N$ which is mapped isomorphically onto 
  $T_{t(g)} \partial \overline{U}$ under the map $T_gt$.  
  Observe that this implies  
  $T_g \partial N =  T_{t(g)}\overline{\sigma} (T_{t(g)} \partial \overline{U}) \oplus \ker T_g t$.
  Since $s\circ \overline{\sigma}$ is a diffeomorphism, 
  $T_g s$ is injective on $T_{t(g)}\overline{\sigma} (T_{t(g)} \partial \overline{U})$.
  Hence the claim is proved, if we can show that $\ker T_g s \cap \ker T_g t$
  is a proper subspace of $\ker T_g t$. Let us prove this by contradiction, and assume that 
  $\ker T_g s = \ker T_g t$. Since $T_gt$ maps $\ker T_g s$ to the image of the anchor map at 
  $t(g)$, the assumption  $\ker T_g s = \ker T_g t$ implies that the image of the anchor map at 
  $t(g)$ is zero. As a consequence, the orbit through  $t(g)$ has to be discrete which contradicts our 
  original assumption. Hence  $\ker T_g s \cap \ker T_g t \subsetneq \ker T_g t$, which finishes the proof that $s_{|\partial N}$ is a submersion.

  Besides $s_{|N}$, the restricted source map 
  $s_{|t^{-1} (U) \cap s^{-1} (U)} : t^{-1} (U) \cap s^{-1} (U) \rightarrow U$
  now has to be locally trivial as well, with typical fiber given by  
  $t^{-1} (U) \cap s^{-1} (p)$. This entails that  $\sfG_{|U}$ is source-locally trivial.  
  Since by construction, $\calO \cap U$ is diffeomorphic to an open ball in some 
  euclidean space, $\calO \cap U$ is a manifold of finite type, which proves the final 
  claim.
\end{proof}
Now let us come to the slice theorem for proper Lie groupoids and recall the two 
original versions by Zung and Weinstein.
\begin{theorem}[{cf.~\cite[Thm.~2.3]{Zun}}]
\label{Thm:SliceThm}
  Any proper Lie groupoid $\sfG$ with a fixed point $p\in \sfG_0$ is locally isomorphic
  to a linear action groupoid, namely the action groupoid of the action of the compact
  isotropy group $\sfG_p$ on the tangent space $T_p\sfG_0$.
\end{theorem}
\begin{theorem}[{cf.~\cite[Thm.~9.1]{w:linear2}}] 
\label{Thm:SliceThmZuWei}
  Assume that $\sfG$ is a proper Lie groupoid, and assume
  that $\calO \subset \sfG_0$ is an orbit of finite type, i.e.~diffeomorphic as a manifold to the interior 
  of a compact manifold with boundary. Then there
  is an open neighborhood $U$ of $\calO$ in $\sfG_0$ such that the restriction
  $\sfG_{|U}$ of $\sfG$ to $U$ is isomorphic as a Lie groupoid to the restriction
  of $\sfG_{|\calO} \ltimes N\calO$ to an appropriate  neighborhood of the
  zero section in $N\calO$.
\end{theorem}
\begin{remark}
\label{Rem:spectub}  
Theorem \ref{Thm:SliceThm} is nowadays known as Zung's Theorem. Although the  
theorem is stated as above in \cite{Zun}, the arguments in that paper  
actually assume source-local triviality. 
However, Proposition 7.3.8 by Dufour and Zung \cite{DufZunPSNF} 
(see also Proposition \ref{Prop:sloctriv} above) together with Zung's original proof in \cite{Zun} show that 
Theorem \ref{Thm:SliceThm} holds true in full generality. 
Let us also mention that a direct and independent proof of Zung's Theorem
(without the source-local triviality assumption used in the proof) has been given in 
\cite{cs}, to which we also refer for further discussion about this issue.

Theorem \ref{Thm:SliceThmZuWei} is a direct consequence of Zung's Theorem and \cite[Thm.~9.1]{w:linear2}.
It also follows from \cite[Thm.~1]{cs}.

The finite type assumption for the orbit in Theorem \ref{Thm:SliceThmZuWei} is an assumption which might not hold in an intended
application. Nevertheless, for an arbitrary proper Lie groupoid $\sfG$ one can conclude
from  Prop.~\ref{Prop:sloctriv} and Thm.~\ref{Thm:SliceThmZuWei} that for every point $p\in \sfG_0$ 
there is an open neighborhood $U$ of $p$ such that the restriction $\sfG_{|U}$ is source-locally trivial, and
isomorphic  to the restriction of $\sfG_{|O} \ltimes NO$ to an open neighborhood of the zero section in $NO$, where 
$O:= U \cap \calO$, and $\calO$ is the orbit through  $p$. We will make frequent use
of this fact in the remainder of this article. In particular, this
observation will be a crucial ingredient in the proof of the following result.
\end{remark}

\begin{proposition}
\label{equivtofinite}
 For every proper Lie groupoid $\sfG$  there exists an
 open embedding of Lie groupoids $\iota: \sfH \hookrightarrow \sfG$
 such that $\iota$  is a weak equivalence and such that
 every orbit of $\sfH$ is of finite type.
\end{proposition}
\begin{proof}
  Choose for each point $\calQ$ of the orbit space $X=\sfG_0/\sfG$
  a preimage $q $ under the canonical projection
  $\sfG_0 \rightarrow X$. Then choose a sufficiently small
  neighborhood $B_{q}$ of $q$ in the orbit $\calQ$
  with the following two properties:
  \begin{enumerate}
  \item The closure $\overline{B_{q}}$ is diffeomorphic to a closed bounded
        ball in some euclidean space.
  \item There exists an open neighborhood $U_{q}$
        of $q$ in $\sfG_0$ and an isomorphism of Lie groupoids
        $\varphi_{q} :
        \sfG_{| B_{q}} \ltimes \Tu_{q} \rightarrow \sfG_{|U_{q}}$
        between the restriction of $\sfG_{| B_{q}} \ltimes NB_{q}$
        to an open invariant tubular neighborhood $\Tu_{q}$
        of the zero section of $NB_{q}$ and the restricted Lie groupoid
        $\sfG_{| U_{q}}$.
  \end{enumerate}
  Observe that by construction each orbit in
  $\sfG_{| B_{q}} \ltimes \Tu_{q}$ is of finite type.
  After possibly shrinking the neighborhoods $U_{q}$ one can select
  a countable subset $Q \subset X$  such that the family
  $\big( U_{q}/\sfG \big)_{q\in Q} $ is a locally finite open covering
  of $X$. Now put
  $\sfH_0 := \bigcup_{q \in Q} U_{q}$ and $\sfH:= \sfG_{|\sfH_0}$.
  By construction, the canonical embedding $\sfH \hookrightarrow \sfG$
  is a weak equivalence. Moreover, each orbit in $\sfH_0$ is of finite type,
  since the orbits of the groupoids $\sfG_{| B_{q}} \ltimes \Tu_{q}$
  are of finite type and since the covering
  $\big( U_{q}/\sfG \big)_{q\in Q} $ is locally finite.
\end{proof}
For the remainder of this section we collect some further results about proper
Lie groupoids that will be needed in the course of the paper.
\begin{proposition}
\label{loc-morita}
  Let $p \in \sfG_0$ a point in a Lie groupoid $\sfG$, and let
  $Z\subset \sfG_0$ be a submanifold which intersects orbits transversally.  Let $\operatorname{Sat}(Z) := \orbproj^{-1} \orbproj (Z)$ be the saturation of $Z$ in
  $\sfG_0$. Then the embedding
  \[
    \sfG_{|Z} \hookrightarrow \sfG_{|\operatorname{Sat}(Z)}
  \]
  is a weak equivalence.
\end{proposition}
\begin{proof}
First remark that $\operatorname{Sat}(Z)$ is open in $\sfG_0$, so $\sfG_{|\operatorname{Sat}(Z)}$ is indeed a Lie groupoid.
As stated above, we have to check the two conditions (ES) and (FF) for the inclusion.
The first follows from the fact that
\[
  \left(\sfG_{|\operatorname{Sat}(Z)}\right)_1\times_{\operatorname{Sat}(Z)}Z:=
  \{g\in\sfG_1\mid s(g)\in Z\},
\]
which clearly shows that the map $t$ to $\operatorname{Sat}(Z)$ is surjective.
Since in a Lie groupoid the target map is a submersion, transversality of $Z$ implies that
the restriction of $t$ to $\operatorname{Sat}(Z)$ is a submersion as well. Secondly, (FF) is
obvious by the definition of $\operatorname{Sat}(Z)$.
\end{proof}
\begin{definition}
A slice at $p\in\sfG_0$ is a submanifold $Z\subset \sfG_0$  with $p\in Z$, which intersects orbits transversally, and $T_p\sfG_0=T_pZ\oplus T_p\calO_p$.
\end{definition}
\begin{proposition}
\label{Prop:Bisec}
  Let $p, p' \in \sfG_0$ be two objects which lie in the same orbit $\calO$ of
  a proper Lie groupoid $\sfG$ .
  Let $Z$ and $Z'$ be two submanifolds of $\sfG_0$ such that $p \in Z$,
  $p'\in Z'$ and which are both  slices  to $\calO$.
  Then, after possibly shrinking $Z$ and $Z'$ there exists an open neighborhood
  $U\subset \sfG_0$ of $p'$ with $U \cap Z' = Z'$ and a local bisection
  $\sigma :U \rightarrow \sfG_1$ such that $t \circ \sigma_{|Z'}$ is a
  diffeomorphism from $Z'$ onto $Z$.
\end{proposition}
\begin{proof}
  Since there exists an arrow $g\in \sfG_1$ with $s (g) =p$ and $t(g)=p'$
  and a local bisection $\sigma$ around $p$ such that $\sigma (p) =g$, one can reduce
  the claim to the case, where $p=p'$. Moreover, since the claim is local it suffices
  by remark  \ref{Rem:spectub} to assume that $\sfG$ is isomorphic to
  $\sfG_{\calO} \ltimes N\calO$ where $\calO$ is the orbit through $p$.
  Finally, by an immediate composition argument, one can reduce the claim to the case,
  where $Z'$ is an open neighborhood of the origin in $N_p$. So let us show this.
  Choose a local section $\varrho :O \rightarrow \sfG(p,-) := s^{-1} (p)$ of
  $t_{|\sfG(p,-)}$ such that $O$ is an open neighborhood of $p$  in the orbit $\calO$.
  Then consider the map
  \begin{displaymath}
    \Psi: O \times N_p \rightarrow NO, \quad (q,n) \mapsto \varrho (q) \cdot n .
  \end{displaymath}
  Since $T_{(p,0_p)} \Psi = \id$, and $\Psi$ is linear in each fiber, $\Psi$ becomes an
  isomorphism of vector bundles after shrinking $O$. Now recall that $Z \subset NO$
  is transversal to $O$, hence after shrinking $Z$, the projection map
  \begin{displaymath}
    P:Z \rightarrow N_p, \quad  P := \operatorname{pr}_2 \circ \Psi^{-1}_{|Z}
  \end{displaymath}
  becomes an open embedding. Let $V_p := P(Z)$. After shrinking $O$, one can assume
  that $O$ is diffeomorphic to the unit ball in some $\R^N$ such that $p$ corresponds
  to the origin under this diffeomorphism. Denote by $O_{\frac 12}$ the image of the
  ball of radius $\frac 12$ under this diffeomorphism. Addition in $\R^N$ gives rise
  to a smooth map
  \begin{displaymath}
    \star : O_{\frac 12} \times O_{\frac 12} \rightarrow O .
  \end{displaymath}
  Note that then $p \star p = p$ and that $p$ is a neutral element with respect to
  $\star$. Put $U_p := \Psi \Big( O_{\frac 12} \times
    \big( V_p \cap P ( \pi^{-1} (O_{\frac 12}) \cap Z ) \big )\Big)$, where $\pi$
  is the projection for $NO$, and define
  \begin{displaymath}
    \begin{split}
      \sigma : \, & U_p \rightarrow \big( \sfG_{|O} \ltimes NO \big)_1 \subset \sfG_1, \\
      & q \mapsto
      \Big( \varrho \big( \operatorname{pr}_1 (q) 
      \star \pi \big( P^{-1} (\operatorname{pr}_2 (q))\big)\big), \operatorname{pr}_2 (q) \Big) 
      \circ \Big(\varrho ( \operatorname{pr}_1 (q), \operatorname{pr}_2 (q) \Big)^{-1} ,
    \end{split}
  \end{displaymath}
  where $\operatorname{pr}_1 (q) = \pi (q) \in O$, and $\operatorname{pr}_2 (q) \in N_p$ is
  the projection onto the second coordinate under the isomorphism $O\times N_p \cong NO$.
  We check now two properties, namely Eq.~\eqref{eq1} and  Eq.~\eqref{eq2}.
  For $n\in N_p$, which implies in particular $\pi (n) =p$ and
  $\operatorname{pr}_2 (n)=n$, we get with $b:= P^{-1} (n)$:
     \begin{equation}
     \label{eq1}
        \begin{split}
            t\circ \sigma (n)\, &  = \varrho (\pi (b)) \cdot n = \varrho (\pi (b))\cdot P(b) =
            \Psi \big( \pi (b), P(b) \big) = \\
            & = \Psi \Big( \operatorname{pr}_1 \big( \Psi^{-1} (b)\big),
            \operatorname{pr}_2 \circ \Psi^{-1} (b) \Big) = b = P^{-1} (n)  .
        \end{split}
     \end{equation}
 Moreover, for $q \in O_{\frac 12}$ one has
   \begin{equation}
     \label{eq2}
        \begin{split}
           t \circ \sigma (q) = \varrho (q) \cdot 0_p = t (\varrho (q)) = q .
        \end{split}
    \end{equation}
 Hence, $T_p (t\circ \varrho)$ is invertible, which implies that after shrinking $U_p$,
 the   map $t\circ \sigma \rightarrow \sfG_0$ is an open embedding. Finally,
 for $q \in U_p$,
 \begin{displaymath}
   s \circ \sigma (q) = \varrho ( \operatorname{pr}_1  (q) ) \cdot \operatorname{pr}_2 (q) =
   \Psi \big(\operatorname{pr}_1  (q) , \operatorname{pr}_2 (q) \big) = q.
 \end{displaymath}
 Hence, $\sigma$ is a bisection with the desired properties.
\end{proof}

\begin{remark}
The properness assumption in the preceding proposition is not necessary
(cf.~\cite{Cra:pc}).  In the case of a Lie
algebroid, a similar result as above has been proved in 
\cite[Thm 1.2]{fernandes}. Fernandes' proof can be easily adjusted to 
give another proof of Prop.~\ref{Prop:Bisec} without the properness 
assumption. We have chosen the above proof since we need it for our later 
construction of the metric tubular neighborhood in 
Theorem \ref{Thm:SliceThm2}. More precisely, we shall use the 
following corollary to the proof of Prop.~\ref{Prop:Bisec}.
\end{remark}

\begin{corollary}\label{cor:local}
Let $p\in \sfG_0$ be a point in an orbit $\calO_p$ of a proper Lie groupoid $\sfG$. Then there
is a neighborhood $U$ of $p$ in $\sfG_0$ diffeomorphic to $O\times
V_p$, where $O$ is an open ball in $\calO_p$ centered at $p$, and
$V_p$ is a $G_p$ invariant open ball in $N_p\calO_p$ centered at the
origin. Under this diffeomorphism, the restricted groupoid
$\sfG_{|U}$ is isomorphic to the product of the pair groupoid
$O\times O\rightrightarrows O$ and the groupoid
$G_p \ltimes V_p\rightrightarrows V_p $.
\end{corollary}

We end this section by proving the existence of complete, transversally invariant 
riemannian metrics for proper Lie groupoids.
\begin{definition}
 Given a riemannian metric $\eta$ on $G_0$, we shall say that it is {\em transversally invariant},
 if, restricted to each orbit $\calO$, the induced metric on the normal bundle
 $N_\calO$ is invariant under the canonical action of $\sfG|_\calO$.
\end{definition}
\begin{remark}
 Riemannian structures on Lie groupoids have been considered before, cf.~\cite{gghr}, but 
there it is assumed that both $\sfG_0$ and $\sfG_1$ carry riemannian metrics which are 
compatible in some sense. More recently, another version of riemannian groupoids was 
introduced in \cite{gl}, where only $\sfG_0$ is required to be equipped with a riemannian 
metric, subject to the condition that the local bisections of $\sfG_1$ act by isometries. 
Our definition of a transversally invariant metric in turn is weaker than this: one easily 
shows the metric on $\sfG_0$ on a riemannian groupoid (in either of the concepts 
mentioned above) is transversally invariant.
\end{remark}
\begin{proposition}
\label{Prop:InvMet}
 Let $\sfG$ be a proper Lie groupoid. Then there exists a transversally invariant, 
 complete, riemannian metric on $\sfG_0$.
\end{proposition}
\begin{proof}
As remarked above, $\sfG$ acts for every orbit $\calO$ on the normal bundle
$N_\calO$ to the orbit, but does in general not act on the tangent bundle $T\sfG_0$.
However, there is a way to lift the representations on the normal bundles to a
non-associative action on $T\sfG_0$, part of a {\em representation up to homotopy}
as in \cite{ac}, called the adjoint representation.
This construction depends on a choice of an Ehresmann connection $\sigma$ on
$\sfG$, that is, a smooth choice of a splitting $\sigma_s:s^*T\sfG_0\to T\sfG_1$
of the exact sequence of vector bundles
\[
 0\longrightarrow t^*A \stackrel{r}{\longrightarrow} T\sfG_1
 \stackrel{ds}{\longrightarrow} s^*T\sfG_0\longrightarrow 0,
\]
which reduces to the canonical inclusion over $\sfG_0$. Here,
$A:=\ker ds_{|\sfG_0}$ is the Lie algebroid of $\sfG$, and $r$ is the map given by
right multiplication on $\sfG_1$. Such a splitting amounts to the choice of a distribution on $\sfG_1$ 
complementary to the $s$-fibers, which equals $T\sfG_0$ at the identity. 
In turn, composition with the inversion map 
determines a splitting
$\sigma_t:=\sigma_s\circ d\iota:t^*T\sfG_0\to T\sfG_1$ of the similar exact sequence 
\[
 0\longrightarrow s^*A \stackrel{l}{\longrightarrow} T\sfG_1
 \stackrel{dt}{\longrightarrow} t^*T\sfG_0\longrightarrow 0,
\]
with the roles of $s$
and $t$ reversed, where $l$ is now induced by the left multiplication.

With this choice of Ehresmann connection, we can define smooth sections
$\mu\in\Gamma^\infty(\sfG_1,\Hom(s^*A,t^*A))$ and
$\lambda\in\Gamma^\infty(\sfG_1,\Hom(s^*T\sfG_0,t^*T\sfG_0))$ by
\[
\begin{split}
  \mu&:=-(\pi_t\circ l),\quad \text{and}\\
  \lambda&:=dt\circ \sigma_s,
\end{split}
\]
where $\pi_t:T\sfG_1\to t^*A$ is the projection induced by $\sigma_s$, and $l$ denotes left multiplication in $\sfG_1$. We think of $\mu$ and $\lambda$ as giving for each $g\in \sfG_1$ a homomorphism, denoted $\mu_g, \lambda_g$, from $A_{s(g)}$ resp. $T_{s(g)}\sfG_0$ to $A_{t(g)}$ resp. $T_{t(g)}\sfG_0$. In general these sections will not define representations of $\sfG$, e.g. $\lambda_{g_1}\circ\lambda_{g_2}\neq \lambda_{g_1g_2}$ for $(g_1,g_2)\in\sfG_2$, but one easily checks that the anchor map $\rho:A\to T\sfG_0$ defined by $\rho:=dt_{|\sfG_0}$ is equivariant with respect to the quasi-actions defined by $\lambda$ and $\mu$. It therefore induces a quasi-action of $\sfG$ on the normal bundle
$N$, i.e., the quotient $T\sfG_0\slash\rho(A)$. This action is the same as the canonical one defined in \eqref{rep-normal-bundle} because the Ehresmann connection exactly defines a choice of lifting smooth paths in $\sfG_0$ to $\sfG_1$, and therefore the quasi-action on $N$ induced by $\lambda$ and $\mu$ is in fact a true action. In other words, the obstruction to associativity of the quasi-action on $T\sfG_0$  lies in the image of the anchor, a precise formula can be found in \cite{ac},  and we can view $\lambda$ as a smooth, albeit non-associative, lifting of the canonical action of $\sfG$ on $N$. In the following proof, we will use the dual quasi-action of $\sfG$ on $T^*\sfG_0$, i.e. 
\[
\lambda^{*}_g: T^*_{t(g)}\sfG_0\to T^*_{s(g)}\sfG_0. 
\]
Let $N^\perp\subset T^*\sfG_0$ be the conormal bundle of $\rho(A)\subset T\sfG_0$. The ``bundle" $N^\perp$ is canonically identified with the dual bundle of the normal bundle $N=T\sfG_0\slash\rho(A)$. As $\lambda$ preserves $\rho(A)$, the quasi-action $\lambda^*$ preserves $N^\perp$. Furthermore, the restricted $\sfG$ quasi-action $\lambda^*_{|N^\perp}$ agrees with the canonical $\sfG$ action on the dual of $N$, and therefore is associative. 

Next we fix a smooth Haar system $\big(\mu^x\big)_{x\in\sfG_0}$ and choose a
``cut-off'' function $c$ for $\sfG$. This is a smooth function on $\sfG_0$
with the properties
\begin{enumerate}[(CU1)]
\item $s:\operatorname{supp}(c\circ t)\to\sfG_0$ is proper,
\item
\[
 \int_{s^{-1}(x)}c(t(g))d\mu^x(g)=1\quad\text{for all}~x\in\sfG_0.
\]
\end{enumerate}
Such a function exists by properness of $\sfG$, cf.~\cite{tu}.
Given an arbitrary riemannian metric $\tilde{\eta}$ on $T^*\sfG_0$, we can now consider the averaged metric given by
\[
  \eta'_x(\alpha,\beta)=\int_{s^{-1}(x)}\tilde{\eta}_{t(g)}\big(
  \lambda^*_{g^{-1}}(\alpha),\lambda^*_{g^{-1}}(\beta)\big) \, c(t(g))d\mu^x(g),
\]
where $\alpha,\beta\in T^*_x\sfG_0$. It is not immediately clear that this indeed defines a metric, more specifically, is positive definite. However, since the Ehresmann connection is required to coincide with $T\sfG_0$ when restricted to $\sfG_0$, it follows that each identity arrow $1_x,~x\in\sfG_0$ acts by the identity transformation. By continuity therefore, $\tilde{\eta}_{t(g)}(\lambda^*_{g^{-1}}(\alpha),\lambda^*_{g^{-1}}(\alpha))> 0$ for $g$ in a neighborhood of $1_x$ in $s^{-1}(x)$, and we see that $\eta'$ is positive definite. When restricted to the conormal bundle $N^\perp$, the averaged metric $\eta'_{|N^\perp}$ is invariant under $\sfG$ by the invariance of the system of measures $(c\circ t)\mu^x$. Let $\eta^*$ be the dual metric of $\eta'$ on $T\sfG_0$. It is easy to check that the induced metric on the normal bundle $N$ is dual to the restricted metric $\eta'_{|N^\perp}$ via the canonical pairing between $N$ and $N^\perp$. As $\eta'_{|N^\perp} $ is $\sfG$-invariant, the induced metric on $N$ is also $\sfG$-invariant. Therefore, $\eta^*$ is transversally invariant.

This proves the existence of a transversally invariant riemannian metric, and we finally 
show that it can even be made complete: for this we adapt the simple argument given by 
\cite[Lemma A5]{w:linear2} to the groupoid case. Using averaging as above, on easily 
shows that there exist $\sfG_0$-invariant proper smooth functions on $\sfG_0$. Choosing 
such a function $f$, consider the metric
\[
\eta:=\left(1+||\operatorname{grad}(f)||^2_{\eta^*}\right)^{-1}\eta^*.
\]
With respect to $\eta$, we then have $||\operatorname{grad}(f)||_\eta\leq 1$, and therefore $\eta$ is complete.
\end{proof}
\begin{remark}\label{rmk:rui}
In the above proof of the existence of a transversally invariant metric, the step of passing to the dual bundle was suggested to us by Matias del Hoyo and Rui Fernandes \cite{hofe}. 
\end{remark}

\section{The  slice theorem for proper Lie groupoids revisited}
\label{sec:metric-slice}
Assume that we are given a proper Lie groupoid $\sfG$ with a
transversally invariant riemannian metric on $\sfG_0$. Let $\calO \subset \sfG_0$ be
an orbit and $p\in \calO$ a point. Choose $\varepsilon :\calO
\rightarrow \R_{>0}$ such that the metric exponential map $\exp$
defines a tubular neighborhood $(N\calO,\varepsilon,\varphi)$ of
$\calO$ in $M$, where $N\calO$
is the normal bundle to $\calO$ in $M$ and $\varphi :=
\exp_{|\Tu^\varepsilon_{\calO,N\calO}} : \Tu^\varepsilon_{\calO,
N\calO} \rightarrow \Tu_\calO := \exp (\Tu^\varepsilon_{\calO,
N\calO})$. Let $\pi_\calO :\Tu_\calO \rightarrow \calO$ be the projection of the 
tubular neighborhood. Denote by $B_{\varepsilon (p)} (0_p) = N_p \cap
\Tu^\varepsilon_{\calO, N\calO}$ the ball of radius $\varepsilon (p)
$ in $N_p\calO$ and by $ Z_p$ the slice $\exp ( B_{\varepsilon (p)}
(0_p))$. We can assume without restriction now and in the following
that $\varepsilon (q) \leq \varepsilon (p)$ for all $q \in \calO$.
Moreover, by Prop.~\ref{Prop:Bisec}, one can shrink $\varepsilon$ so
that one  has for the projection $\orbproj : \sfG_0 \rightarrow X$ onto
the orbit space:
\begin{equation}
\label{Eq:ProjProp}
  \orbproj (Z_q) \subset \orbproj (Z_p)
  \quad \text{for all $q\in \calO$}.
\end{equation}
After possibly shrinking $\varepsilon$ again, there exists a
homomorphism $\lambda : \sfG_{|Z_p} \rightarrow \sfG_p$ by \cite[Sec.~2.2]{Zun} which induces
an isomorphism of groupoids $ \theta : \sfG_p  \ltimes Z_p \rightarrow \sfG_{|Z_p} $
whose inverse is $(\lambda,s_{\sfG_{|Z_p}})$.
Since the restricted exponential map
$\exp_{|B_{\varepsilon (p)} (0_p)} : B_{\varepsilon (p)} (0_p) \rightarrow Z_p$
is a $\sfG_p$-equivariant isomorphism, one obtains an isomorphism of groupoids
$\Theta : \sfG_p \ltimes B_{\varepsilon (p)} \rightarrow
\sfG_{|Z_p} $ by $\Theta := \theta \circ (\id_{|\sfG_p},\exp_{|B_{\varepsilon
(p)} (0_p)})$. One concludes that the following diagrams commute.
\begin{equation}
  \label{GrpdHomDiag1}
    \xymatrix{\sfG_p \ltimes B_{\varepsilon (p)}
    \ar[r]^{s}\ar[d]_\Theta & \ar[d]^{\exp} B_{\varepsilon (p)}\\
    \sfG_{|Z_p} \ar[r]_{s} & Z_p}
    \qquad
    \xymatrix{\sfG_p \ltimes B_{\varepsilon (p)}
    \ar[r]^{t}\ar[d]_\Theta & \ar[d]^{\exp} B_{\varepsilon (p)}\\
    \sfG_{|Z_p}  \ar[r]_{t} & Z_p}
\end{equation}

Next consider the submanifold $\sfG (-,Z_p) := t^{-1} (Z_p)$ of
$\sfG_1$. As explained in Step 2. of the proof of
\cite[Thm.~9.1]{w:linear2}, the isotropy group $\sfG_p$ acts on
$\sfG (-,Z_p)$ via the embedding of $\sfG_p$ into the group of
bisections of $\sfG_{|Z_p}$. The restriction $t_{|\sfG (-,Z_p)} :
\sfG (-,Z_p) \rightarrow Z_p$ of the target map then is a
$\sfG_p$-equivariant submersion. By Thm.~\ref{thm:mettub} there
exists a continuous map $\tilde \varepsilon   : \sfG(-,p)
\rightarrow \R_{>0}$ and a $\sfG_p$-equivariant tubular neighborhood
$\big( E , \tilde\varepsilon, \Phi \big)$ of $\sfG (-,p)$ in
$\sfG(-,Z_p)$ such that the compatibility relation
\begin{equation}
\label{Eq:comprelS}
   \pi_\calO s \, ( \pi_{\sfG(-,p)} (\gamma )) =  \pi_\calO s \, (\gamma )
\end{equation}
holds for all $\gamma \in \Tu^{\tilde\varepsilon}_{\sfG(-,p)}$ and
$\pi_{\sfG(-,p)} : \Tu^{\tilde\varepsilon}_{\sfG(-,p)} \rightarrow
\sfG(-,p)$ the projection of the tubular neighborhood.
Moreover, one can achieve after shrinking $\tilde\varepsilon$ that
\[
  \big( \pi_{\sfG(-,p)} , t_{|\Tu^{\tilde\varepsilon}_{\sfG(-,p)} } \big):
  \Tu^{\tilde\varepsilon}_{\sfG(-,p)} \rightarrow \sfG(-,p) \times Z_p
\]
is an open embedding. Since $\tilde\varepsilon$ is $\sfG_p$-invariant
one can even assume after further shrinking $\varepsilon$ and
$\tilde\varepsilon$ that $\tilde\varepsilon = \varepsilon \circ s_{|\sfG(_,p)} $. By slight abuse of
language we will therefore write $\varepsilon$ instead of $\tilde\varepsilon$
from now on. The restricted source map
$s_{|\Tu^\varepsilon_{\sfG (-,p)}} : \Tu^\varepsilon_{\sfG (-,p)} \rightarrow
     \Tu^\varepsilon_{\calO}$ is then a surjective submersion.
Now we will use Step 3 in the proof of \cite[Thm.~9.1]{w:linear2} to construct
a submersive retraction $\Pi : \sfG_{|\Tu^\varepsilon_\calO} \rightarrow
\sfG_{|\calO}$. To this end let $\gamma \in \sfG_{|\Tu^\varepsilon_\calO}$.
Since then $s(\gamma) \in \Tu^\varepsilon_\calO$ and $t(\gamma) \in
\Tu^\varepsilon_\calO$, Eq.~(\ref{Eq:ProjProp}) entails that there are arrows
$\iota , \kappa \in
\Tu^{\varepsilon}_{\sfG (-,p)} $ such that $s(\iota) = t(\gamma)$ and $s(\kappa)
= s(\gamma)$. The composition $\iota \gamma \kappa^{-1} $ then is well-defined
and lies in $\sfG_{|Z_p}$. Hence on can put
\begin{equation}
  \Pi (\gamma) :=  \pi_{\sfG (-,p)} (\iota)^{-1} \,  \lambda \big( \iota \gamma
  \kappa^{-1} \big) \, \pi_{\sfG (-,p)} (\kappa) .
\end{equation}
By Eq.~\eqref{Eq:comprelS} and since $\pi_\calO$ is a retraction onto $\calO$
one obtains for the target and source of $\Pi(\gamma)$
\begin{equation}
\label{Eq:Compst}
\begin{split}
  s \big( \Pi (\gamma ) \big) & \, = s \big( \pi_{\sfG (-,p)} (\kappa) \big) =
  \pi_\calO s \pi_{\sfG (-,p)} (\kappa) =  \pi_\calO s (\kappa) =
  \pi_\calO s (\gamma) , \text{ and }\\
  t \big( \Pi (\gamma ) \big) & \, = s \big( \pi_{\sfG (-,p)} (\iota) \big) =
  \pi_\calO s \pi_{\sfG (-,p)} (\iota) =  \pi_\calO s (\iota) =
  \pi_\calO t (\gamma).
\end{split}
\end{equation}
Hence $\Pi (\gamma ) \in \sfG_{|\calO}$ indeed.
As has been shown in Step 3 of the proof of \cite[Thm.~9.1]{w:linear2},
$\Pi (g)$ does not depend on the choices made, and is a retraction
onto $\sfG_{|\calO}$. Moreover, it has been shown there that
$\Pi (\gamma \, \gamma') = \Pi (\gamma) \, \Pi (\gamma')$ for
$\gamma,\gamma ' \in \sfG_{|\Tu^\varepsilon_\calO}$ with
$s(\gamma) = t (\gamma')$. Together with Eq.~\eqref{Eq:Compst}
this shows that $\Pi$ is a morphism of groupoids over the
map $\pi_\calO$ on objects. Now we define
\begin{equation}
\label{Eq:DefGrpdHom}
   \Theta := \big( \Pi , \exp_{|\Tu^\varepsilon_\calO}^{-1} \circ s \big) :
   \sfG_{|\Tu^\varepsilon_\calO} \rightarrow
   \big( \sfG_{|\calO} \ltimes N\calO \big)_{|\Tu^\varepsilon_{\calO,N\calO}} .
\end{equation}
Let us show that after possibly shrinking $\varepsilon$, $\Theta$
is an isomorphism of groupoids.   To this end let $K\subset \calO$ be a compact
submanifold with boundary of dimension $\dim \calO$. Then $K^\circ$ is a
manifold of finite type. Put $\varepsilon_K := \min\limits_{q\in K} \varepsilon
(q)$. Then by Step 4 of the proof of \cite[Thm.~9.1]{w:linear2},
the map
\[
  \Theta_K : \sfG_{| \Tu^{\varepsilon_K}_{K^\circ}} \rightarrow
  \big( \sfG_{|K^\circ} \ltimes N K^\circ \big)_{\Tu^{\varepsilon_K}_{K^\circ,N
  K^\circ}}
\]
is an open embedding. Observe that $\calO$ can be exhausted by a countable
number of compact manifolds with boundary $K_n$, $n\in \N$ such that $K_n^\circ$ is open
in $\calO$. After having defined the $\varepsilon_{K_n}$ as above shrink
$\varepsilon$ such that $\varepsilon (q) \leq \max\limits_{ \{ n  \in \N \mid q \in K_n^\circ\} }
\varepsilon_{K_n} (q)$. One checks immediately that then $\Theta$ has to be a
diffeomorphism. According to \cite[Thm.~9.1]{w:linear2} each of the
$\Theta_K$ is a morphism of groupoids, hence $\Theta$ has to be a morphism of
groupoids as well. This proves the following result:

\begin{theorem}
\label{Thm:SliceThm2}
  Let $\calO$ be an orbit in a proper Lie groupoid $\sfG$. Then there  
  exists a neighborhood $U$ of $\calO$ in $\sfG_0$ such that $\sfG_{|U}$ is isomorphic to a  
  neighborhood of the zero section of $\sfG_{|\calO}\ltimes N\calO$. More precisely:
  Let $\eta$ be a transversally invariant riemannian metric
  on $\sfG_0$.  Then there exists a
  continuous map $\varepsilon : \calO \rightarrow \R_{>0}$ and a submersive
  retraction $\Pi: \sfG_{|\Tu^\varepsilon_\calO}  \rightarrow \sfG_{|\calO}$
  such that $\Theta: \sfG_{|\Tu^\varepsilon_\calO} \rightarrow
   \big( \sfG_{|\calO} \ltimes N\calO \big)_{|\Tu^\varepsilon_{\calO,N\calO}}$
 as defined by Eq.~\eqref{Eq:DefGrpdHom} is an isomorphism of groupoids,
 and such that the following two diagrams commute:
 \begin{equation}
  \label{GrpdHomDiag2}
    \xymatrix{\sfG_{|\Tu^\varepsilon_\calO}
    \ar[r]^{s}& \Tu^\varepsilon_\calO\\
    \big( \sfG_{|\calO} \ltimes N\calO
    \big)_{|\Tu^\varepsilon_{\calO,N\calO}}  \ar[u]_{\Theta^{-1}}  \ar[r]_{\qquad s} &
    \ar[u]_{\exp_{|\Tu^\varepsilon_{\calO,N\calO}}}
    \Tu^\varepsilon_{\calO,N\calO}}
    \quad
    \xymatrix{\sfG_{|\Tu^\varepsilon_\calO}
    \ar[r]^{t}& \Tu^\varepsilon_\calO\\
    \big( \sfG_{|\calO} \ltimes N\calO
    \big)_{|\Tu^\varepsilon_{\calO,N\calO}} \ar[u]_{\Theta^{-1}} \ar[r]_{\qquad t}
    & \ar[u]_{\exp_{|\Tu^\varepsilon_{\calO,N\calO}}}
    \Tu^\varepsilon_{\calO,N\calO}}
\end{equation}
In case there is an open embedding $\sfH \hookrightarrow \sfG$ 
of proper Lie groupoids such that the closure of $\calO \cap \sfH_0$ 
is a compact subset, then the function 
$\varepsilon :\calO \rightarrow \R_{>0}$ can be chosen to be constant over 
$\calO \cap \sfH_0$.
\end{theorem}
\begin{remark}
In general, the tubular neighborhood  $\Tu^\varepsilon_\calO$ is
{\em not} $\sfG$-invariant. However, in  the case where 
$\varepsilon$ can be chosen to be constant, it follows that
$\Tu^\varepsilon_\calO$ is invariant. This is an important
distinction between such orbits and the general case. See
\cite{cs} for an interesting and careful discussion.
\end{remark}

\section{The canonical stratification of proper Lie groupoids}
\label{sec:stratification}
\subsection{Stratification}
\label{stratification}
Now consider a point $p \in \sfG_0$ in the orbit $\calO$, and choose a
slice $Z\subset \sfG_0$ to the orbit $\calO$ at $p$.
After possibly shrinking $Z$, there exists, by Zung's theorem stated above, an
isomorphism of groupoids
\[
 \varphi: \sfG_{|Z}\rightarrow \sfG_p\ltimes V_p,
\]
with $V_p$ a neighborhood of the origin in $N_p\calO$. This isomorphism induces
in particular an action of $\sfG_p$ on $Z$. Let $S_{Z,p}:= Z^{\sfG_p}$  be the fixed
point manifold of $\sfG_p$ in $Z$, and let $\calS_p$ denote the set germ of the
projection of $S_{Z,p}$ in $X$, that means let
\begin{equation}
  \label{Eq:DefStrat}
  \calS_p := [\orbproj (S_{Z,p}) ]_\calO .
\end{equation}
Then we have the following results.
\begin{lemma}\label{lem:slices}
  The set germ $\calS_p$ does not depend on the particular choice of the slice
  $Z$ to the orbit $\calO$ at $p$ and the isomorphism
  $\varphi: \sfG_{|Z}\rightarrow \sfG_p\ltimes V_p$.
\end{lemma}
\begin{proof}
Let $Z'$ be a second submanifold through $p$ which is transversal to the orbit $\calO$
through $p$. Choose a local bisection $\sigma : U \rightarrow \sfG$ according to
Proposition \ref{Prop:Bisec} such that $t\circ \sigma$ is a diffeomorphism from $Z$
onto $Z'$. We then can define a morphism of groupoids
$\Phi : \sfG_{|Z} \rightarrow \sfG_{|Z'}$ by putting
\begin{displaymath}
  \Phi (g) := \big( \sigma (t(g)) \big) g \big( \sigma (s (g) \big)^{-1} .
\end{displaymath}
One checks immediately, that $\Phi$ is an isomorphism of groupoids, and that
$\Phi_0 $ is given by
\begin{displaymath}
  Z \ni q \mapsto t (\sigma (q))\in Z'.
\end{displaymath}

In the next step, we use Zung's Theorem again to obtain an isomorphism of groupoids
$\varphi': \sfG_{|Z'}\rightarrow \sfG_p\ltimes V'_p$, where $V_p'$  is another
$\sfG_p$-invariant neighborhood of the origin in the normal space $N_p \calO$.
Then consider the composition
\[
  \Psi := \varphi' \circ \Phi \circ \varphi^{-1} :
  \sfG_p\ltimes V'_p \rightarrow \sfG_p\ltimes V'_p .
\]
By construction, $\Phi$ is an isomorphism of groupoids. Let $ q \in S_{Z,p} $. Then
$\varphi (q)$ is an element of the invariant space $V_p^{\sfG_p}$. Hence
$\Psi (\varphi (q)) = \varphi' (\Phi (q)) $ is an element of the invariant space
$(V_p')^{\sfG_p}$, and consequently $\Phi (q) \in S_{Z',p} $. Since $q$ and
$\Phi (q)$ lie in the same orbit, one obtains 
$\orbproj (S_{Z,p}) \subset \orbproj  (S_{Z',p})$.
By symmetry, we have $\orbproj  (S_{Z',p}) \subset \orbproj  (S_{Z,p})$. 
The claim follows.
\end{proof}
\begin{lemma}\label{lem:point}
  If $p,q$ are in the same orbit of $\sfG$, then $\calS_p = \calS_q$.
\end{lemma}
\begin{proof}
Any $g\in\sfG_1$ with $s(g)=p$ and $t(g)=q$ defines an isomorphism
$\sfG_p\cong\sfG_q$ by conjugation.
This isomorphism is independent of $g\in s^{-1}(p)\cap t^{-1}(q)$ up to conjugation
in $\sfG_q$ (or $\sfG_p$).
Furthermore, the element $g$ induces a $\sfG_p$-$\sfG_q$ equivariant isomorphism from
$N_p\calO$ onto $N_q\calO$. Therefore, in the quotient the two set germs $\calS_p$
and $\calS_q$ are canonically isomorphic.
\end{proof}
 By the last lemma, $\calS_p$ depends only on the projection $\orbproj (p)$, hence 
 we will identify  $\calS_\calO$ for $\calO\in X$ with $\calS_p$ where 
 $\orbproj (p)=\calO$.
 \begin{theorem}
 \label{strat-orbit}
   The mapping which associates to every point $\calO\in X$ the set germ
   $\calS_\calO$ in $X$  is a stratification of $X$.
 \end{theorem}
\begin{proof}
  Let $\calO_p \in X =\sfG_0/\sfG$. Choose a slice $Z\subset \sfG_0$
  to the orbit $\calO_p$ at $p$ together with an isomorphism of groupoids
  $\varphi : \sfG_{|Z} \rightarrow \sfG_p \ltimes V_p$ as above. It is well-known
  that the orbit space $V_p/\sfG_p \cong Z/\sfG_{|Z} $ has a canonical stratification
  by orbit types. Denoting by $\varrho : V_p \rightarrow V_p/\sfG_p$ the canonical
  projection, this stratification is given by
  \begin{displaymath}
    \calT_{\varrho (v)} = \big[ \varrho (V_p^{(\sfG_{p,v})}\big]_{\varrho (v)},
  \end{displaymath}
  where $v \in V_p$, $\sfG_{p,v} \subset \sfG_p$ is the isotropy group of $v$,
  and $V_p^{(\sfG_{p,v})}$ is the subspace of all elements of $V_p$ having an
  isotropy group conjugate to $\sfG_{p,v}$.
  The claim is thus proven if under the isomorphism $V_p/\sfG_p \cong Z/\sfG_{|Z} $
  we can show that $\calT_{\varrho ( \varphi (q))} = \calS_{\orbproj (q)}$ for all $q\in Z$.
  To this end denote first by $\sfG_q$ the isotropy group of $q$. Since $\sfG_{|Z}$
  is isomorphic to $V_p\rtimes \sfG_p$, the isotropy group $\sfH$ can be identified
  with a subgroup $\sfH$ of $\sfG_p$. Note that this identification depends on
  $\varphi$. Let $Z^{(\sfH)}$ be the space of all elements
  $b\in Z$ having isotropy group conjugate to $\sfH$. Now choose a submanifold
  $D\subset Z$ through $q$ which is transversal to the $\sfG_p$-orbit through $q$.
  Note that by the slice theorem for group actions one can choose $D$ in such a way
  that $D$ is invariant under $\sfH$ and such that $D\cap Z^{\sfH} = V \cap Z^{\sfH}$
  for some open neighborhood $V\subset Z$ of $q$. By construction it is
  clear that $D$ is also transversal to the $\sfG$-orbit through $q$ in $\sfG_0$.
  Moreover,
  \[
    D^{\sfG_q}  = D^{\sfH} = D \cap Z^{\sfH} = V \cap Z^{\sfH} ,
  \]
  hence we obtain for the set germs at $\orbproj (q)$:
  \[
   \calS_{\orbproj (q)} = \big[\orbproj (D^{\sfG_q})\big]_{\orbproj (q)} =
   \big[ \orbproj (Z^{\sfH}) \big]_{\orbproj (q)}
   \cong \big[ \varrho (V_p^{\sfH}) \big]_{\varrho (\varphi (q))} =
   \big[ \varrho (V_p^{(\sfH)}) \big]_{\varrho (\varphi (q))}
   = \calT_{\varrho (\varphi (q))}.
  \]
  Thus the claim follows.
\end{proof}

\begin{corollary}\label{cor:stratification}
  The quotient space $X$ of a proper Lie groupoid $\sfG$ carries in a natural
  way the structure of a stratified space with smooth structure. Moreover, $X$
  is even Whitney stratified and has a system of smooth control data.
\end{corollary}
\begin{proof}
  According to the preceding theorem, the previously defined $\calS$ is a
  natural stratification of $X$. Let us show, that $(X,\calS)$ even has a
  canonical smooth structure. For every open $U\subset X$ let
  $\calC^\infty_X (U)$ be the space
  $\calC^\infty_\text{\tiny\rm inv} \big(\orbproj^{-1} (U)\big)  $ of
  $\sfG$-invariant smooth functions on $\orbproj^{-1} (U)$. Obviously, we thus get a
  sheaf on $X$. Let us show that $\calC^\infty_X$ is the structure sheaf of a
  smooth structure on $X$. To this end let $\calO_p \in X$ be an orbit, and choose
  a slice $Z \subset \sfG_0$ to the orbit $\calO_p$ at $p\in \sfG_0$.
  After possibly shrinking $Z$, $\orbproj (Z) $ is open in
  $X$. Consider the map
  \[
    r_Z : \calC^\infty_X \big( \orbproj (Z) \big) \rightarrow \calC^\infty (Z), \quad
    f \mapsto f\circ \orbproj_{|Z} .
  \]
  We show that after possibly shrinking $Z$ again,  $r_Z$ induces an isomorphism
  \begin{equation}
    \label{eq:locfuncstruc}
    \calC^\infty_X \big( \orbproj (Z) \big) \cong \big( \calC^\infty (Z) \big)^{\sfG_p} .
  \end{equation}
  By Prop.~\ref{loc-morita}, $r_Z$ is injective. To prove surjectivity,
  first choose an open neighborhood $U$ of $p$ according to Remark
  \ref{Rem:spectub}, i.e.~such that $\sfG_{|U}$ is isomorphic to the restriction
  of the transformation groupoid $\sfG_{|O} \ltimes NO$ to an open neighborhood
  of $O:= U \cap \calO_p$ in $NO$. After possibly shrinking $Z$ and $U$, one
  can achieve that $Z=U\cap Z$. By Prop.~\ref{Prop:Bisec} we can assume, after
  possibly shrinking $Z$ and $U$ again, that $Z \subset N_pO$. Shrinking $U$
  further, one can achieve that for every $q \in U$ there is an arrow
  $g_q \in \sfG_{|O}$ such that $g_q \cdot q \in Z$. Now
  let $h \in \big( \calC^\infty (Z) \big)^{\sfG_p}$.
  Define $H: U \rightarrow \R$ by $H(q) = h (g_q \cdot q)$. Then, by invariance
  of $h$, $H(q)$ is independent of the particular choice of $g_q$. Moreover,
  $H$ is smooth by construction. Next, let $q \in \operatorname{Sat} (Z)$, and
  choose a bisection $\sigma_q :V \rightarrow \sfG_1$ over an open
  neighborhood $V \subset \sfG_0$  such that $t (\sigma (V)) \subset U$.
  Define $f: \operatorname{Sat} (Z) \rightarrow \R$ by $f(q) = H(t(\sigma (q)))$.
  By construction, $f$ then is smooth, independent of the particular choices
  of the bisections $\sigma_q$ and invariant. Moreover, $f_{|Z} = h$.
  This proves surjectivity of $r_Z$.

  From the work \cite{Schwa:SFIACLG} one concludes that there exists a
  singular chart $\chx : Z/\sfG_p \rightarrow \R^n$ such that
  $\big( \calC^\infty (Z) \big)^{\sfG_p} \cong \chx^* \calC^\infty (W)$, where
  $W\subset \R^n$ is open with $\chx ( Z/\sfG_p ) \subset W$ being closed
  (see \cite{pf:book} for details). Moreover, the stratified space
  $\chx ( Z/\sfG_p )$ satisfies Whitney's condition B. This entails that
  $X$ carries a canonical smooth structure which satisfies
  Whitney's condition B in each singular chart. By Thm.~\ref{SaExKoDaWhi},
  $X$ is controllable. This proves the claim.
\end{proof}
\begin{remark}
We end with several remarks about this stratification:
\begin{itemize}
\item[$i)$]
  If $p\in \sfG_0$, the set germs $\calT_p := \orbproj^{-1} (\calS_p)$ define
  a stratification of the object space $\sfG_0$.
\item[$ii)$] In case $\sfG$ is \'etale, the quotient space $X$ is an orbifold and the stratification of Theorem \ref{strat-orbit} is the usual one.
\item[$iii)$] On the other extreme, if $\sfG$ is $s$-connected, collecting strata of the same dimension, one obtains a coarser stratification from Theorem \ref{strat-orbit} which agrees with the one of the singular riemannian foliation of Proposition \ref{prop-lie-SRF}, constructed in \cite[\S 6.2]{MolRF}. This follows from the fact that the orbits of the singular riemannian foliation induced by the groupoid $\sfG$ are given in terms of Lie-algebroid paths for $A$ which can be used to integrate the groupoid $\sfG$, cf.~\cite{cf}.
\end{itemize}
\end{remark}

\subsection{Orbit types}
In Theorem \ref{strat-orbit}, we have defined a stratifications of
$X$ and of $\sfG_0$ using set-germs. This local and conceptually 
very clear approach to stratification theory originates in the work of 
Mather. It has the advantage of avoiding the technical constructions of 
partitions of $X$ and $\sfG_0$ which define the stratification. In
the case of a proper $G$-action on a manifold $M$ (c.f.~Example 
\ref{ex:action}), there is a natural way to define the orbit type 
stratification of $M$ and also of $X=M/G$ by a partition
with respect to a certain ``orbit type" equivalence relation 
(cf.~\cite{dk:book}). 
Inspired by this construction we briefly discuss in this subsection
a natural decomposition of $\sfG_0$ and $X$ by ``orbit type".

\begin{definition}
\label{dfn:normal-orbit-type}
Let $\sfG$ be a proper Lie groupoid.
\begin{itemize}
\item[$(i)$] The {\em weak orbit type} of a point $p\in\sfG_0$ is the (abstract) isomorphism class of the isotropy group $\sfG_p$: We write $p\thicksim_o q$ if $\sfG_p\cong\sfG_q$ as Lie groups.
\item[$(ii)$] The {\em normal orbit type} of $p\in\sfG_0$ is the isomorphism class of the representation of $\sfG_p$ on  $N_p\calO$: we write $p\sim_n q$ if $p\thicksim_o q$ and the isomorphism $\sfG_p\cong\sfG_q$ extends to an isomorphism $N_p\calO\cong N_q\calO$ of representations.
\end{itemize}

\end{definition}
The adjective "weak" is put in $(i)$ to indicate that in the case of a proper Lie group action, 
this definition is a weaker version of the usual notion of orbit type, see below.
Clearly, both $\thicksim_o$ and $\thicksim_n$ define equivalence
relations on $\sfG_0$ and partition $\sfG_0$ into subsets of
$\sfG_0$. As two points on the same orbit have the same weak orbit type
and normal orbit type (Lemma \ref{lem:slices} and \ref{lem:point}),
$\thicksim_o$ and $\thicksim_n$ define also equivalence relations on
the orbit space $X=\sfG_0/\sfG$. Since the isomorphism classes of the
orbit type and normal orbit type are invariant under Morita
equivalence, the partitions on $X$ defined by $\thicksim_o$ and
$\thicksim_n$ are also Morita invariant. The main result proved in this 
section about these partitions is the following:
\begin{theorem}
\label{main-part}
Let $\sfG$ be a proper Lie groupoid.
\begin{itemize}
\item[$i)$] Each normal orbit type is an open and closed subset of its weak orbit type.
\item[$ii)$] The connected components of the partition into weak orbit types are locally closed submanifolds of $\sfG_0$ and form a decomposition of $\sfG_0$ in the sense of Definition \ref{dfn:decomposition}.
\item[$iii)$] On the level of set germs, the induced stratification of this decomposition agrees with the one defined in section \ref{stratification}.
\end{itemize}
\end{theorem}

\subsubsection{Proper group action}
\label{pga}
Let us consider the action groupoid $G\ltimes M$ for a proper $G$-action
on $M$. Then one can define two natural equivalence relations 
on $M$. Put for $p,q\in M$
\begin{enumerate}
\item $p\thicksim q$, if $G_p$ is conjugate to $G_q$ in G, and
\item $p\thickapprox q$, if there is a $G$-equivariant diffeomorphism
$\Phi$ from an open $G$-invariant neighborhood $U_p$ onto an open
$G$-invariant neighborhood $U_q$ of $q$ such that $\Phi (p)=q$.
\end{enumerate}

The equivalence relation $p\thicksim q$ is stronger than $p\thicksim_o
q$ as there are isomorphisms between subgroups which are not
conjugation by an element in $G$. The following proposition determines
the relation between the partitions on $M$ defined by $\thicksim$
and $\thicksim_o$.
\begin{proposition}\label{prop:orbit-type}
For $p\in M$, there is an open neighborhood $U$ of $p$ such that in $U$ 
the relation $p\thicksim_o q$ holds true if and only if $p\thicksim q$.
\end{proposition}

For the proof, we will need the following well known property about subgroups of a
Lie group $G$:
\begin{lemma}\label{lem:group}
 For a compact subgroup $H$ of $G$,
 every closed subgroup $H_0$ of $H$ that is isomorphic to $H$ as Lie
 groups is identical to $H$.
\end{lemma}
\begin{proof}
We refer the reader to the arguments in the proof of \cite[Lemma
4.2.9]{pf:book}.
\end{proof}

\begin{proof}[Proof of Proposition \ref{prop:orbit-type}]
According to the slice theorem for group actions there 
exists a $G$-invariant open neighborhood $U$ of $p$ and a $G$-equivariant
diffeomorphism $U \rightarrow G\times_{G_p} B$, where 
$B$ is a $G_p$-invariant open neighborhood of $0$ in $N_p\calO_p$.
Hence we can assume without loss of generality that $M = G\times_{G_p} B$,
$p= [e,0]$ with $e$ being the identity element in $G$, and $q=[g,x]$ for some 
$g\in G$ and $x\in B$. Obviously, $p\thicksim q$ implies $p\thicksim_o q$. 
It remains to show that $p\thicksim_o q$ implies $p\thicksim q$.
Since $q$ lies in the same $G$-orbit as $[e,x]$, $G_q$ is conjugate to 
$G_{[e,x]}$. But $G_{[e,x]}$ is a closed subgroup of $G_p$. Since $G_p$ is 
compact, Lemma \ref{lem:group} implies that $G_{[e,x]}$ is identical to 
$G_p$. Therefore, $p\thicksim q$.
\end{proof}

\begin{remark}A similar result holds true for $\thicksim_n$ and $\thickapprox$ with the similar arguments as Proposition \ref{prop:orbit-type}. 
\end{remark}
\begin{proposition}\label{prop:strat-group-action}
The connected components of the partition defined by $\thicksim_o$ define 
the orbit type decomposition on $M$, which induce a stratification of the 
quotient $X=M/G$.
\end{proposition}
\begin{proof}
Proposition \ref{prop:orbit-type} implies that the connected components
of the partition defined by $\thicksim$ is identical to those
defined by $\thicksim_o$. This proposition follows from the similar
result for the equivalence relation $\thicksim$ as in
\cite[Thm.2.7.4]{dk:book}.
\end{proof}
\subsubsection{The decomposition of $\sfG_0$}
Now let us consider the case of a proper Lie groupoid and the equivalence
relations $\thicksim_o$ and $\thicksim_n$. 
Considering for $p\in \sfG_0$ the neighborhoods $O\times V_p$
described in Corollary \ref{cor:local} one concludes that the connected 
components of the partition defined by $\thicksim_o$ are submanifolds. 
Furthermore, Proposition \ref{prop:strat-group-action} implies the following
result:
\begin{proposition}\label{prop:partition}
The germs of the partition
defined by the orbit type $\thicksim_o$ on $\sfG_0$ and $X$ agree
with $\calS$ as defined by (\ref{Eq:DefStrat}).
\end{proposition}
Theorem \ref{main-part} $iii)$ follows immediately from this Proposition.
To prove that the partition by the normal orbit type defines a
natural decomposition of $\sfG_0$ and $X$, we need the following two lemmata.
\begin{lemma}\label{lem:component}
For every $p\in \sfG_0$ there exists a neighborhood $U$ of $p$ such that for 
$q\in U$ the relation $q\thicksim_o p$ holds true if and only if $q\thicksim_n p$.
\end{lemma}
\begin{proof}
It suffices to prove that $q\thicksim_o p$ implies
$q\thicksim_n p$. By Corollary \ref{cor:local}, this follows from
the analog statement for the case of a compact group action on a
manifold, which is proved in \cite[Thm 2.6.7 (i)]{dk:book}.
\end{proof}
Lemma \ref{lem:component} shows that the normal orbit types are open and closed in the orbit types, and  that the normal orbit type decomposes $\sfG_0$ into submanifolds, cf.~Theorem \ref{main-part} $i)$.
\begin{lemma}\label{lem:normal-orbit-type} If $p,q\in \sfG_0$
such that $p\thicksim_n q$, then there exist open neighborhoods $U_p$ and $U_q$ 
of $p$ and $q$ in $\sfG_0$ such that the restricted groupoids $\sfG_{|U_p}$
and $\sfG_{|U_q}$ are isomorphic.
\end{lemma}
\begin{proof}
This follows from Corollary \ref{cor:local} which describes $\sfG$ locally.
\end{proof}
It remains to prove Theorem \ref{main-part} $ii)$:
\begin{proposition}\label{thm:stratification}
 The connected components of the partitions by orbit type $\thicksim_o$ (and therefore also of normal orbit types $\thicksim_n$) define decompositions of 
 $\sfG_0$ and $X$. 
\end{proposition}
\begin{proof}
Lemma \ref{lem:component} proves that the connected components of the partition
by normal orbit type are submanifolds of $\sfG_0$. The only not immediate
property of  a decomposition which remains to be shown is the
condition of frontier.
To prove the condition of frontier  consider two  connected components 
$R$ and $S$ of the partition  by normal orbit type.
We have to show that if $R\cap \overline{S}\ne \emptyset$,
then $R\subset \overline{S}$. Let $p\in R \cap \overline{S}$. 
By Lemma \ref{lem:normal-orbit-type}, there exist for $q\in R$ open
neighborhoods $U_p$ and $U_q$ of $p$ and $q$ respectively such that 
$\sfG_{|U_p}$ is isomorphic to $\sfG_{U_{q}}$. This implies that $q$ is also in 
$\overline{S}$. Hence $R\subset \overline{S}$.
%
\end{proof}
\section{Metric properties of proper Lie groupoids}
\label{sec:metric}
In this section we discuss the metric properties of the quotient space.
Let $\sfG$ be a proper Lie groupoid equipped with a transversally invariant riemannian metric.
There is a standard way, cf.~\cite[Sec.~$1.16^+$]{GroMSRNRS} and \cite[Def.~3.1.12]{BurBurIvaCMG},
to define a semi-metric on the quotient by
\begin{equation}
\label{semi-metric-G}
\bar{d}(\calO,\calO'):= \inf \big\{ d(q_1 , \calO) + \ldots +
    d (q_k , \calO_{k-1} ) \big\},
\end{equation}
where the infimum is taken over all choices of pairs $(q_i,\calO_i),~q_i\in\calO_i,~ i=1,\ldots,k-1$ and   $q_k\in\calO'$, over all $k\in\N$. Our main result is as follows:
\begin{theorem}
\label{Thm:LocMetSrucOrbitspace}
 Let $\sfG$ be a proper Lie groupoid such that the orbit space $X$ is
 connected. Let $\eta$ be a transversally invariant riemannian metric on $\sfG$. Then the induced
semi-metric \eqref{semi-metric-G} on the  orbit space $X$ is in fact a metric and the following
 properties hold true:
 \begin{enumerate}[{\rm (1)}]
   \item The metric $\overline{d}$ is uniquely determined by the property
         that for each orbit $\calO$ in $\sfG_0$ and every point $q \in \Tu_\calO$
         of an appropriate metric tubular neighborhood $\Tu_\calO$ of $\calO$
         the relation
         \begin{equation*}
            \overline{d} (\calO, \calO_q ) = d (q ,\calO)
         \end{equation*}
         holds true, where $\calO_q$ is the orbit through $q$.
   \item The canonical projection $\orbproj : \sfG_0 \rightarrow X$
         onto the orbit space is a submetry, i.e.~every
         ball $B_r (p)$ in $(\sfG_0,\eta)$ with respect to the geodesic
         distance on the riemannian manifold $(\sfG_0,\eta)$ is mapped under
         $\orbproj$ onto the ball $B_r (\calO_p)$ in $X$.
   \item $(X,\overline{d})$ is a length space, i.e.~the geodesic distance
         with respect to $\overline{d}$ coincides with $\overline{d}$.
         Moreover, the topology induced on $X$ by $\overline{d}$ coincides with
         quotient topology with respect to $\orbproj$.
   \item
         In case $(\sfG_0,\eta)$ is a complete riemannian manifold,
         $(X,\overline{d})$ is even a complete locally compact length space,
         and every bounded closed ball is compact.
   \item In case $(\sfG_0,\eta)$ has curvature bounded from below,
         $(X,\overline{d})$ is an Alexandrov space globally of dimension
         $\leq \dim \sfG_0$.
\end{enumerate}

\end{theorem}
\begin{remark}
This theorem generalizes the following special case: let $f:M\to N$ be a  submersion.
Associated is a proper Lie groupoid $\sfG$ with $\sfG_0:=M$, $\sfG_1:=M\times_NM$, 
source and target are induced by the two projections to $M$, and composition is given 
by $(x,y)\cdot (y,z)=(x,z)$.
A transversally invariant riemannian metric on $\sfG$ is a metric on $M$ which descends to a riemannian metric on $N$, and makes $f$ into a riemannian submersion. Since $N$ is the quotient space of this particular groupoid and $f$ the projection map, this fact suffices to prove Thm \ref{Thm:LocMetSrucOrbitspace} in this case.

For a general proper Lie groupoid, with singular quotient space, a
submetry is exactly the right generalization of a riemannian
submersion.
  For more information on the metric concepts used in the theorem
  see the monograph \cite{BurBurIvaCMG}. Length spaces are covered in Chapter 2
  of that book, Alexandrov spaces in Chapter 10.
  Alexandrov spaces are essentially length spaces with curvature bounded
  below. They generalize riemannian manifolds with curvature bounded below
  but share  many properties with riemannian manifolds.
\end{remark}
\subsection{Proper Lie groupoids and singular riemannian foliations}
Before proving the above theorem let us introduce some
notation, and prove several auxiliary results. First, assume to be given a
smooth manifold $M$ with a riemannian metric $\eta$, and a submanifold $S\subset
M$, not necessarily closed. Let $N \subset T_{|S}M$ denote the normal bundle to
$S$ in $M$, and $\pi^N : N \rightarrow S$ its projection. Then there exists
a continuous function $\varepsilon_S : S \rightarrow \R_{>0}$ such that for
each continuous $r : S \rightarrow \R_{>0}$ with $0 < r \leq \varepsilon_S $
the restriction of the exponential map to
\begin{equation}
 U_{S,r} := \big\{ v \in N  \mid \| v \| < r \big( \pi^N (v) \big) \big\}
\end{equation}
is a diffeomorphism onto its image which will be denoted by $T^r_S$. Choosing
such an $\varepsilon_S$, each of the sets  $T^r_S$ with $0< r \leq  \varepsilon_S$
then is a (metric) tubular neighborhood of $S$ in $M$, and
\begin{equation}
   T^r_S =  \exp \big( U_{S,r} \big) = \big\{ p \in M \mid d (p, S) < r
   \big( \pi_S (p) \big) \},
\end{equation}
where $d(p,S)$ is the geodesic distance from $p$ to $S$, and
$\pi_S : T^r_S \rightarrow S$ is the projection of the tubular neighborhood
$T^r_S$. In a situation, where $\exp_{|U_{S,r}}$ is a diffeomorphism onto its
image for all $0 < r \leq \varepsilon_S $, we say that $\varepsilon_S $ defines
a tubular neighborhood of $S$ in $M$. In case $S$ is a compact submanifold,
$\varepsilon_S$ can be chosen to be constant, i.e.~one can choose
$\varepsilon_S >0$ to be a positive real number such that the constant function
on $S$ with value $\varepsilon_S$  defines a tubular neighborhood of $S$ in $M$.
After these notational preparations, let us proceed now to prove the auxiliary
results.
\begin{lemma}
\label{Lem:orth}
  Let $(M,\eta)$ be a connected riemannian manifold, and $S\subset M$ be a
  closed submanifold. Let $p\in M$ be a point, and denote by
  $d(p,S)$ the geodesic distance between $p$ and $S$ which means the
  infimum of all geodesic distances $d(p,q)$, where $q$ runs through the
  elements of $S$. Assume that there is a point $q\in S$ such that
  $d(p,q) =d(p,S)$, and that $d(p,q)$ is less than the convexity radius of $M$
  at $p$ and at $q$. Then the tangent vector $X := \exp_q^{-1} (p) \in T_qM$
  is orthogonal to $T_qS$.
\end{lemma}
\begin{proof}
  Let $Y\in T_qS$ be a tangent vector, and
  $\gamma : (-\delta, \delta) \rightarrow S$ a smooth path in $S$
  with $\gamma (0)=q$ and $\dot{\gamma}=Y$ .
  Then, after possibly passing to some smaller $\delta$, the tangent vectors
  $\widetilde{X}(t) := \exp_p^{-1} \big( \gamma (t) \big)$ are well-defined for
  $|t|<\delta$. Since $\eta \big( \widetilde{X}(t),\widetilde{X}(t) \big)$
  has a local minimum  at $t=0$, one obtains
  \[
     0 = \eta \Big( \widetilde{X}(0), \big( T_{\widetilde{X}(0)}
     \exp_p\big)^{-1} Y \Big).
  \]
  By the Gauss-Lemma, $Y$ is orthogonal to
  $T_{\widetilde{X}(0)} \exp_p \big( \widetilde{X}(0) \big)$, and
  $T_{\widetilde{X}(0)} \exp_p \big( \widetilde{X}(0) \big) = - X$.
  The claim follows.
\end{proof}
Let us now recall, following \cite[Chap.~6]{MolRF}, the notion of a 
{\em singular riemannian foliation} on a manifold $M$: 
this is a partition of $M$ by connected immersed submanifolds (the leaves) satisfying
\begin{itemize}
\item[$i$)] (singular foliation) the module of vector fields tangent to the leaves acts transitively on each leaf, 
\item[$ii)$] (riemannian structure) there exists a riemannian metric on $M$ with the property that every geodesic 
that is perpendicular to a leave at one point, remains perpendicular to every leaf it meets.
\end{itemize}
A metric as in $ii)$ above, is called {\em adapted} to the singular foliation. 
A typical example of a singular riemannian foliation is given by the orbit foliation of an isometric 
action of a Lie group on a riemannian manifold.

The following result is stated in \cite[Corollary
7.4.13]{DufZunPSNF} without a proof.
\begin{proposition}
\label{prop-lie-SRF}
  Let $\sfG$ be a proper Lie groupoid, and $\eta$ a transversally invariant 
  riemannian  metric on $\sfG_0$. Then the connected components of the orbits 
  of $\sfG$ define a singular riemannian foliation on $\sfG_0$ for which 
  $\eta$ is an adapted metric.
  \end{proposition}
\begin{proof}
  Consider the map $D$ which associates to each point $p\in \sfG_0$ the tangent
  space $T_p\calO_p$, where $\calO_p$ denotes the orbit through $p$.
  This is a singular vector-distribution in the sense of
  Stefan--Sussmann \cite{SteASOFS,SusOFVFID,DufZunPSNF} which can be identified
  with the image under the anchor map $\rho$ of the Lie algebroid $A$ of $\sfG$.
  For any Lie algebroid $(A,\rho)$ it is known, cf.\ \cite[\S 8.1.4]{DufZunPSNF} 
  that $D=\rho(A)$ defines a singular foliation: first of all, by the local splitting 
  theorem \cite[Thm 1.1]{fernandes}, the distribution is smooth.
  Second, the maximal integral submanifolds $L$ of $D$ can be
  described in terms of ``$A$-paths'' (cf.\ \cite{fernandes}) as
  follows:  we say that $x,~y\in\sfG_0$ are $A$-equivalent, written
  $x\sim_Ay$, if there exists an $A$-path connecting the two points.
  This defines an equivalence relation whose orbits are immersed
  submanifolds of $\sfG_0$, in fact embedded in case $\sfG$ is proper.
  They are the  maximal integrating submanifolds of $D$, so the
  Stefan--Sussmann Theorem \cite[Thm.~1.5.1]{DufZunPSNF} implies that
  $D$ is the tangent distribution of a singular foliation,

  It remains to be shown that the singular foliation $\sfG_0$ together with
  $\eta$ is riemannian. Since $\eta$ is transversally invariant and the claim is local,
  it suffices by \cite[Prop.~6.4]{MolRF} to show that for each point
  $p\in \sfG_0$ there exists a neighborhood $W$ and a riemannian metric
  $\varrho$ on $W$ such that the pair $(W,\varrho)$ is a singular riemannian
  foliation. To this end choose a neighborhood $O\subset \calO_p$
  diffeomorphic to a ball in $T_p\calO_p$ under the exponential map.
  By the slice theorem for proper Lie groupoids, $\sfG$ looks in a neighborhood
  of $p$ like the transformation groupoid
  $\sfG_{|O} \ltimes \tilde W$, where $\tilde W$ is a $\sfG_{|O}$-invariant
  open neighborhood of the zero-section of the normal bundle $NO$.
  So it suffices to prove that on $\big( \sfG_{|O} \ltimes NO \big)_0$ there is
  an adapted riemannian metric $\varrho$ which means
  that each geodesic which emanates orthogonally to an orbit stays orthogonally
  to every orbit it passes through. By the choice of $O$ one has a
  trivialization $NO \cong O \times N_pO$ (Corollary \ref{cor:local}). Now choose a 
  $\sfG_p$-invariant
  scalar product $\varrho^\textup{v}$ on $N_pO$, and let $\varrho^\textup{h} $
  be the pull-back to $O$ of a euclidean metric on $T_p\calO_p$ under the exponential
  map $\exp$. The product metric
  $\varrho$ of  $\varrho^\textup{v}$ and $\varrho^\textup{h} $ then has the
  desired property. This can be seen as follows. Let $\calQ\subset NO $
  be an orbit of the $\sfG_{|O}$-action on $NO$, and assume that
  $\gamma :[0,1] \rightarrow NO$ is a geodesic which emanates orthogonally
  from $\calQ$. This means that $q :=\gamma (0) \in \calQ$ and
  $\dot{\gamma}(0) \perp T_q \calQ$. Under the trivialization $NO \cong O \times N_pO$
  one has $q =(q_1,q_2)$ with $q_1 \in O$ and $q_2 \in N_pO$. Moreover,
  $\calQ$ then is given under this trivialization by
  $O\times \sfG_p \cdot q_2$, where $\sfG_p$ is the isotropy group of $p$.
  Since $\varrho$ is the product metric of  $\varrho^\textup{v}$ on
  $N_pO$ and $\varrho^\textup{h}$ on $O$,  one concludes that
  \begin{equation}
  \label{Eq:OrthProds}
    O \cong T_q(O\times \{q_2\}) \perp \{ q_1 \} \times N_pO,
  \end{equation}
  hence $\dot{\gamma}(0) \in \{ q_1 \} \times N_pO$. But this implies that
  $\gamma (t)\in \{ q_1 \} \times N_pO$ and $\dot{\gamma} (t) \perp O$
  for all $t\in[0,1]$. By \cite[p.~188]{MolRF}, one obtains that
  $\dot{\gamma} (t)$ is orthogonal in $\{ q_1 \} \times N_pO$ to
  the orbit $\sfG_p\cdot \gamma (t)$. By Eq.~\eqref{Eq:OrthProds} it follows
  that $\dot{\gamma} (t)$ is orthogonal to $T_ {\gamma(t)}\calQ$ for all
  $t\in[0,1]$. This proves that $\varrho$ is an adapted riemannian metric for 
  the singular foliation.
  
Finally, remark that the transversally invariant metric $\eta$ on $\sfG$ induces a 
transverse metric on the normal bundle to each leaf of the foliation as defined in 
\cite[Sec.~6.3]{MolRF}. Therefore, by \cite[Prop.~6.4]{MolRF}, the metric $\eta$ is 
adapted to the foliation.
\end{proof}

\begin{proposition}
  Under the assumptions from above, each connected component of a stratum of
  the stratification of $\sfG_0$ by orbit type is totally geodesic.
\end{proposition}
\begin{proof}
  Let $p\in \sfG_0$, and denote by $S_p$ the stratum through $p$ with respect to the
  stratification of $\sfG_0$ by orbit type.
  Consider the orbit $\calO$ through $p$. Let $\pi^N : N\rightarrow \calO$
  be its normal bundle, and $N^\textup{inv} \rightarrow \calO$ the
  invariant bundle, which by definition consists of the bundles of fibers
  $(N_q)^{\sfG_q}$, where $q\in \calO$
  and $\sfG_q$ is the isotropy group of $q$. Then choose
  a sufficiently small open ball $B$ around $p$ in the orbit $\calO$
  through $p$ such that $B$ is geodesically convex in $\calO$ with respect to the
  riemannian metric induced by $\eta$ on $\calO$. After possibly shrinking $B$,
  there exists an $\varepsilon >0$  such that the metric tubular neighborhood
  $(N_{|B},\varepsilon,\exp_{|\Tu^\varepsilon_{B,N_{|B}}})$ exists.
  Denote by $T^\varepsilon_B$ its total space, and from now on the
  restricted normal bundle $N_{|B}$ shortly by $N$. By the slice theorem for
  proper Lie groupoids as formulated in Thm.~\ref{Thm:SliceThm2},
  one can even choose $B$ and $\varepsilon >0$ such that
  $\sfG_{|B}$ acts on $M:= T^\varepsilon_B$, and such that $\sfG_{|M}$ is isomorphic
  to $\big( \sfG_{|B}\ltimes N \big)_{|\Tu^\varepsilon_{B,N}}$ via an
  isomorphism of groupoids $\Theta$ over the map $\exp^{-1}_{|\Tu^\varepsilon_{B,N}}$ on objects.
  In particular one then has
  \begin{equation}
  \label{Eq:Inv}
      S_p \cap M = \exp \big( \Tu^\varepsilon_{B,N^\textup{inv}} \big)
      \quad \text{where $\Tu^\varepsilon_{B,N^\textup{inv}} :=
      \{ v \in N^\textup{inv} \mid \pi^N (v) \in B \text{ and } \| v \| <
      \varepsilon\} $}.
  \end{equation}
  Since $B$ is in particular contractible,
  there exists a smooth section $\tau : B \rightarrow \sfG (p,-)$ over $B$ of
  the target map $t: \sfG (p,-) \rightarrow \sfG_0$ restricted to morphisms
  with source $p$.   One can even achieve that $\tau (p) = u(p)$.
  Let us now define an action of $\sfG_p$ on $\Tu^\varepsilon_{B,N} $.
  To this end let $v \in \Tu^\varepsilon_{B,N}$, $q := \pi^N (v)
  \in B$, and $g\in \sfG_p$. Then put
  \[
    gv := \big( \tau_q g \tau_q^{-1} \big) v .
  \]
  Clearly, this defines the desired action. Via the isomorphism $\Theta$
  between $\sfG_{|M}$ and
  $\big( \sfG_{|B}\ltimes N \big)_{|\Tu^\varepsilon_{B,N}}$
  the $\sfG_p$-action is transferred to $M$. By Eq.~\eqref{Eq:Inv}
  and the equivariance of the exponential map one obtains
  \[
    S_p \cap M = M^{\sfG_p} .
  \]
  By \cite[Thm.~1.10.15]{KliRG}, the right hand side of this equality is
  totally geodesic.  Since it is enough to show that $S_p$ is locally
  totally geodesic, the claim now follows.
\end{proof}

The next proposition is an equidistance  result for orbits of proper Lie 
groupoids which is crucial for the construction of an appropriate metric
on the orbit space $X$. For singular foliations a corresponding result is 
known \cite{Molorb}. Since we are not aware of an explicit proof 
for the case we consider here, we provide a complete proof below.
\begin{proposition}
\label{Prop:EquiDist}
  Let $\sfG$ be a proper Lie groupoid and $\eta$ a transversally invariant 
  riemannian metric  on $\sfG$. 
  Then the partition of $\sfG_0$ into orbits is \emph{equidistant}
  in the following sense.
  \begin{enumerate}
    \item[{\rm (ED)}]
      For every orbit $\calO$ in $\sfG_0$ there exists a metric
      tubular neighborhood $\Tu_\calO$ such that for every orbit
      $\calO'$ and all $q,q' \in \calO' \cap \Tu_\calO$ the relation
      $d(q,\calO) = d(q',\calO)$ holds true, where $d(q,\calO)$
      denotes the geodesic distance between $q$ and $\calO$.
  \end{enumerate}
\end{proposition}
\begin{proof}
  Let $\calO \subset \sfG_0$ be an orbit of $\sfG$. Let
  $\Tu_\calO^{\varepsilon}$ for some continuous
  $\varepsilon : \calO \rightarrow \R_{>0}$ be a tubular neighborhood
  and $\Theta : \sfG_{|\Tu_\calO^{\varepsilon}} \rightarrow
  \big( \sfG_{|\calO} \ltimes N\calO
  \big)_{|\Tu^{\varepsilon}_{\calO,N\calO}}$ an isomorphism of groupoids
  as in Theorem \ref{Thm:SliceThm2}. Let $q, q'$ be two points in
  $\Tu_\calO^{\varepsilon}$ which both lie in
  some orbit $\calQ$. Then there exists
  an arrow $G\in \sfG_{|\Tu^{\varepsilon}_\calO}$ such that $s(G) =q$ and
  $t(G)=q'$. Let $g = \Pi (G)$ with $\Pi$ as in Theorem  \ref{Thm:SliceThm2}.
  Moreover, let $v,v' \in \Tu^{\varepsilon}_{\calO,N\calO}$ such that
  $q = \exp v$ and $q' = \exp v'$. By Theorem  \ref{Thm:SliceThm2}
  one concludes that $ v' = g v$. Since $\eta$ is a transversally invariant riemannian
  metric, one concludes
  \[
    d(q,\calO) = \| v \| = \| g v \| = \| v' \| = d(q',\calO) .
  \]
  This proves the claim.
\end{proof}
As an application of Proposition \ref{Prop:EquiDist}, we prove a
stability result for  compact orbits of a proper Lie groupoid
$\sfG$ with only connected orbits. We call $\sfG$ \emph{orbit
connected}, if every orbit of $\sfG$ is connected.

Recall that a neighborhood $U$ of an orbit $\calO$ of a groupoid
$\sfG$ is called \emph{invariant}, if every orbit of $\sfG$ which
intersects nontrivially with $U$ is contained in $U$. Following
\cite[Def.~4.2]{cs}, we define an orbit $\calO$ of a Lie groupoid
$\sfG_0$ to be \emph{(topologically) stable}, if $\calO$ admits
arbitrarily small invariant neighborhoods.

\begin{proposition}\label{prop:inv-nbhd} Let $\sfG$ be a proper
Lie groupoid with a transversally invariant riemannian metric $\eta$ on $\sfG_0$.
If all orbits of $\sfG$ are connected, the metric tubular
neighborhood $T^\varepsilon_\calO$ near a compact orbit $\calO$ is
invariant for every sufficiently small constant $\varepsilon>0$. Therefore, a
compact orbit $\calO$ of an orbit connected proper Lie groupoid
$\sfG$ is stable.
\end{proposition}

\begin{proof}
By Proposition \ref{Prop:EquiDist}, we can fix a metric tubular neighborhood
$T^\delta_\calO$ such that property (ED) holds true.

Since $\calO$ is compact, the closure of the metric tubular neighborhood
$T^\varepsilon_\calO$ of $\calO$ is completely contained in $T^\delta_\calO$, if
the constant $\varepsilon>0$ is chosen to be sufficiently small. Furthermore, any
neighborhood of $\calO$ contains a metric tubular neighborhood $T^\varepsilon_\calO$
as a subset. Hence, to prove the proposition, it suffices to show
that for such a small $\varepsilon$, an orbit $\calO'$ of $\sfG$ is completely
contained  in $T^\varepsilon_\calO$, if it intersects $T^\varepsilon_\calO$ nontrivially.

Define $A$ to be the intersection between $\calO'$ and
$T^\varepsilon_\calO$. Since $T^\varepsilon_\calO$ is an open subset
of $\sfG_0$, $A$ is relatively open in $\calO'$. Since $\sfG_0$ is
Hausdorff, $\calO'$ is Hausdorff as well. Let us show that $A$ is
closed in $\calO'$, ormore precisely  that any $x$ in the closure
$\overline{A}$ lies in $A$ as well. So let $x \in \overline{A}$. By
construction of $A$, $x$ is in $\calO'$ and is contained in the
closure of the $\varepsilon$-tubular neighborhood
$T^\varepsilon_\calO$ which is a subset of $T^\delta_\calO$. Since
the closure of $T^\varepsilon_\calO$ is contained in
$T^\delta_\calO$, property (ED) implies that each point in the
intersection between $\calO'$ and $T^\delta_\calO$ has equal
distance to $\calO$. Since $x$ is in $T^\delta_\calO$, $x$ has the
same distance to $\calO$ as any point in $A\subset \calO'$. As
$A\subset T^\varepsilon_\calO$ and the distance function is
continuous, the distance from $x$ to $\calO$ is less than
$\varepsilon$. Therefore, $x$ lies in $ T^\varepsilon_\calO$ and in
$\calO'$, hence one gets $x\in A$. Thus $A$ is both open and closed
in $\calO'$. Since $\calO'$ is connected, $A$ has to coincide with
$\calO'$, and $T^\varepsilon_\calO$ is invariant.
\end{proof}
\begin{remark}
It is natural to expect that a similar result as Proposition
\ref{prop:inv-nbhd} holds true for singular riemannian foliations.
We leave the details to the diligent readers.
\end{remark}
\begin{remark}
Combining Proposition \ref{prop:inv-nbhd} with Theorem
\ref{Thm:SliceThm2}, we conclude that a compact orbit $\calO$ inside
an orbit connected proper Lie groupoid has arbitrarily small
invariant neighborhoods $T^\varepsilon_\calO$ such that
$\sfG_{|T^\varepsilon_\calO}$ is isomorphic to $(\sfG_\calO\ltimes
N\calO)_{|T^\varepsilon_\calO}$. We suggest the reader to compare this
result to Corollary 2 and Remark 1.2 in \cite{cs}.
\end{remark}
\subsection{Proof of the main theorem}
Under the assumptions of the proposition assume now that the orbit
space $X$ is connected. Choose for each point $q\in \sfG_0$
a relatively compact subset $B_q \subset \subset \calO_q$ such that
the closure $\overline{B_q}$ is diffeomorphic to a closed bounded ball
in some euclidean space. Choose $\varepsilon_q >0$ such that
$\Tu_{\overline{B_q}}^{\varepsilon_q}$ is a metric tubular neighborhood
of $\overline{B_q}$ in $\sfG_0$. After possibly shrinking the $\varepsilon_q$
one can achieve that  $\varepsilon_q \leq \varepsilon_{\calO_q} (p)$ for all
$p\in \overline{B_q}$ where $\varepsilon_{\calO_q} :{\calO_q} \rightarrow \R^{>0}$ is
chosen such that $\Pi$ and $\Theta$ as in Theorem  \ref{Thm:SliceThm2}  exist for
the tubular neighborhood $\Tu_{\calO_q}^{\varepsilon_{\calO_q}}$.
After possibly shrinking
the $\varepsilon_q$, one can find by the proof of Proposition
\ref{equivtofinite} a countable subset $Q\subset \sfG_0$ such that
the following holds true:
\begin{enumerate}[(1)]
\item The groupoid $\sfH$ defined by
      $\sfH_0 := \bigcup_{q\in Q}\Tu_{B_q}^{\varepsilon_q}$
      and $\sfH:= \sfG_{|\sfH_0}$ is a proper Lie groupoid with
      orbits of finite type.
\item The canonical embedding $\iota :\sfH \hookrightarrow \sfG$ is
      a weak equivalence.
\item The family $\big( \orbproj (\Tu_{B_q}^{\varepsilon_q}) \big)_{q\in Q}$
      is a locally finite open covering of the orbit space $X$.
\end{enumerate}
Since by $(2)$ the inclusion $\sfH\hookrightarrow\sfG$ is a weak equivalence, they have the same quotient space and we can use $\sfH$ to write down a semi-metric on $X$, instead of $\sfG$.
In fact we can arrange this in the following way: Next let $\calO$ and $\calO'$ be orbits in $\sfG_0$.
Since $X$ is connected, there exist orbits $\calO_0, \cdots,\calO_k$
in $\sfG_0$ and points $q_i \in \calO_i \cap \sfH_0$ for $i=1,\cdots, k$ such
that
  \begin{enumerate}[(CS1)]
    \item $\calO_0 =\calO$ and $\calO_k = \calO'$,
    \item $q_i \in \Tu_{\calO_{i-1}\cap\sfH_0}$ for $i=1,\cdots, k$, where
    $\Tu_{\calO_i \cap\sfH_0} \subset \sfH_0$ is as in the proof of the
    preceding proposition applied to the case of the proper Lie groupoid
    $\sfH$ with riemannian metric $\eta_{|\sfH_0}$.
  \end{enumerate}
  We will call $\calO_0, \cdots,\calO_k, q_1,\cdots, q_k$
  a connecting sequence of orbits and points in $\sfH_0$ between $\calO$ and
  $\calO'$.
  Now we define
  \[
    \overline{d} (\calO, \calO') := \inf \big\{ d(q_1 , \calO_0) + \ldots +
    d (q_k , \calO_{k-1} ) \big\},
  \]
  where the infimum is taken over all connecting sequences of orbits
  and points in $\sfH_0$ between $\calO$ and $\calO'$.
  This semi-metric $\overline{d}$ coincides with the quotient (semi)-metric
  for the canonical projection $\sfH_0 \rightarrow X$ as defined in
  \cite[Sec.~$1.16^+$]{GroMSRNRS} and \cite[Def.~3.1.12]{BurBurIvaCMG}:
  Given any connecting sequence $(\calO_i,q_i)$ not necessarily (CS2) and distance
  minimizing geodesics $\gamma_i$ from $\calO_i$ to $q_{i+1}$, we can realize the same
  sum $\sum_{i}d(q_{i+1},\calO_i)$ by introducing more orbits satisfying (CS2) and subdividing
  the geodesics $\gamma_i$.

  We will show in several steps that $\overline{d}$ is a metric indeed,
  and does not depend on the particular choice of $\sfH$, and in fact equals \eqref{semi-metric-G}. \\[1mm]
  {\bf Step 1.}
  Let $\calO$ be an orbit in $\sfG_0$ and let $q$
  be a point in $\Tu_{\calO \cap\sfH_0}$. Assume that $\calQ$ is the orbit
  through $q$ and that $q'$ is another point of $\calQ$. We will show that
  \begin{equation}
  \label{Eq:Step1}
     d (q,\calO) \leq d (q',\calO) .
  \end{equation}
  To this end it suffices to show that for every path
  $\varrho: [0,\ell] \rightarrow \sfG_0$ parameterized by arc length such that
  $\varrho (0) \in \calO$ and $\varrho (\ell) = q'$ the relation
  $d (q,\calO) \leq \ell$ holds true. Let us show that this is the case indeed.
  We can assume that $\varrho$ is piecewise geodesic and meets every orbit in at
  most one point, otherwise we could pass over to a shorter $\varrho$.
  Recall from above the properties of the continuous maps
  $\varepsilon_{\calO_q } :\calO_q \rightarrow \R_{>0}$, $q \in \sfG_0$.
  After possibly shrinking $\varepsilon_{\calO_p}$ for
  $p \in \varrho([0,\ell])$ one can even
  achieve that $\varepsilon_{\calO_p}$ is constant on the relatively compact set
  $(\calO_p \cap \sfH_0 ) \cup
    \pi_{\calO_p} \big(\Tu_p \cap \varrho([0,\ell])\big)$, where
  $\Tu_p$ is the tubular neighborhood $\Tu^{\varepsilon_{\calO_p}}_{\calO_p}$.
  By compactness of the image of $\varrho$ one can now find
  times $t_0,\cdots,t_l$ such that the following holds true:
  \begin{enumerate}[(1)]
  \item One has $t_0=0$ and $t_l=\ell$.
  \item The restricted paths $\varrho_{|[t_i,t_{i+1}]}$ are geodesics for
        $i=0,\cdots,l-1$.
  \item The image $\varrho ([0,\ell])$ is contained in
        $\bigcup_{i=0}^{l} \Tu_{\calO_{q'_i}}$, where
        $q'_i := \varrho (t_i)$ for $i=0,\cdots,l$, and
  \item $\Tu_{\calO_{q'_i}} \cap \Tu_{\calO_{q'_{i+1}}} \cap \varrho ([0,\ell]) \neq
        \emptyset$ for $i=0,\cdots,l-1$.
  \end{enumerate}
  Put $\calO_i := \calO_{q'_i}$ and $\Tu_i:= \Tu_{\calO_{q'_i}}$. Then choose
  $q_l \in \calO_l\cap\sfH_0$ and points $p'_0,\cdots,p'_{l-1}$ such that
  $p_i'\in \Tu_i \cap \Tu_{i+1} \cap \varrho ([0,\ell])$.
  By Theorem \ref{Thm:SliceThm2} one can now find $g_l\in \sfG_{|\calO_l}$ such that
  $q_l= g_l \pi_l (p'_{l-1}) \in \sfH_0$, where $\pi_i$ is the projection of
  the tubular neighborhood $\Tu_i$. Put $p_{l-1} := g_l p'_{l-1} $.
  Then $p_{l-1}$ lies in $\sfH_0$ and $\Tu_{l-1}$.
  Put $q_{l-1} := \pi_{l-1} (p_{l-1})$.
  Since $p_{l-1}\in \sfH_0$, one has $q_{l-1}\in \calO_{l-1}\cap\sfH_0$.
  Moreover, one has
  \[
  \begin{split}
    d(q_l,q_{l-1})& \leq d (p_{l-1},q_l) + d (p_{l-1},q_{l-1}) =
    d (p_{l-1},\calO_l) + d (p_{l-1},\calO_{l-1}) = \\
    & = d (p'_{l-1},\calO_l) + d (p'_{l-1},\calO_{l-1}) \leq
     d (p'_{l-1}, q_l') + d (p'_{l-1}, q_{l-1}') = t_l - t_{l-1}.
  \end{split}
  \]
  Proceed inductively to construct $q_l,p_{l-1}, \cdots, p_0 , q_0 \in \sfH_0$
  such that
  \[
    d(q_{i-1},q_i) \leq t_i - t_{i-1} \quad \text{for $i = 1,\cdots ,l$}.
  \]
  But then
  \[
    d(q,\calO) \leq \sum_{i=1}^l d(q_{i-1},q_i) \leq \sum_{i=1}^l
    t_i - t_{i-1} = \ell
  \]
  what we wanted to show. Hence  Eq.~\ref{Eq:Step1} holds true.
  \\[1mm]
  {\bf Step 2.}
  Like in Step 1 let  $q\in \Tu_{\calO \cap \sfH_0}$. We will show that for
  each connecting sequence of orbits $\calO_0,\cdots,\calO_k$ and points
  $q_1,\cdots ,q_k$ in $\sfH_0$  between $\calO$ and the orbit $\calO_q$
  through $q$ the estimate
  \begin{equation}
  \label{Eq:est}
    d(q,\calO) \leq  d(q_1,\calO_0) + \ldots + d(q_k,\calO_{k-1})
  \end{equation}
  holds true. This implies in particular that
  \begin{equation}
  \label{Eq:QuotMetProp}
    \overline{d} (\calO, \calO_q ) = d (q ,\calO).
  \end{equation}
  Let us prove \eqref{Eq:est} first in the case where all $q_i$ are  assumed to
  be in $\Tu_{\calO \cap \sfH_0}$. We will prove the claim by induction on
  $k$. Assume first that $k=2$. By construction of the tubular neighborhood
  $\Tu_{\calO \cap \sfH_0}$ there exists an arrow
  $g \in \sfH_{|\calO}$ such that $\pi_1 (q_2) = g q_1$, where $\pi_i$
  for $0\leq i \leq k$ denotes the projection of the tubular neighborhood
  $\Tu_{\calO_i \cap \sfH_0}$. Then $s(g)= \pi_0 (q_1)$. Put
  $p:= t(g)$. By  Proposition \ref{Prop:EquiDist} one concludes
  \[
    d(\pi_1 (q_2) ,p) = d(\pi_1 (q_2) ,\calO_0 ) =
    d(q_1 ,\calO_0) = d(q_1, \pi_0 (q_1) ) .
  \]
  From this one gets
  \[
    d(q,\calO) \leq d(q_2, \pi_1 (q_2) ) + d(\pi_1 (q_2) , \calO_0) =
    d(q_2 ,\calO_1) + d(q_1 ,\calO_0),
  \]
  which proves the claim for $k=2$. Assume that the claim holds for some $k$, and
  consider a connecting sequence of orbits $\calO_0,\cdots,\calO_{k+1}$ and points
  $q_1,\cdots ,q_{k+1}$ in $\sfH_0$  between $\calO$ and the orbit $\calO_q$,
  where it is assumed that all the $q_i$ are in $\Tu_{\calO\cap \sfH_0}$.
  By the inductive assumption and the initial step one gets
  \[
    d(q_1,\calO_0) + \ldots + d(q_k,\calO_{k-1}) + d(q_{k+1},\calO_k) \geq
    d(q_k,\calO_0) + d(q_{k+1},\calO_k) \geq d ( q,\calO) .
  \]
  This finishes the claim for the case, where all points $q_i$ are in
  $\Tu_{\calO \cap\sfH_0}$. If one of the points $q_i$ is not in
  $\Tu_{\calO \cap\sfH_0}$ the claim is obvious.
  \\[1mm]
  {\bf Step 3.}
  Next we will show that $\overline{d}$ is non-degenerate, i.e.~that for orbits
  $\calO$ and $\tilde\calO$ with $ \overline{d} (\calO,\tilde\calO) =0$
  the relation $\calO = \tilde\calO$ holds true. So let us assume that
  $ \overline{d} (\calO,\tilde\calO) =0$. In case there is some
  $q \in \tilde\calO \cap \Tu_{\calO \cap \sfH_0}$, Step 2 entails that
  \[
   d (q,\calO) = \overline{d}(\calO,\tilde\calO)=0 ,
  \]
  which implies that $ q \in \calO$ and thus $\tilde\calO = \calO$.
  Now consider the case $\tilde\calO \cap \Tu_{\calO\cap\sfH_0}= \emptyset$.
  We will show that in this case  $ \overline{d} (\calO,\tilde\calO) =0$
  does not hold true.
  Let $\calO_0,\cdots,\calO_k$, $q_1,\cdots ,q_k$ be a connecting sequence of
  orbits and points in $\sfH_0$  between $\calO$ and the orbit $\tilde\calO$.
  Let $\varepsilon >0$ such that
  $\Tu_{\calO\cap\sfH_0} = \Tu_{\calO\cap\sfH_0}^\varepsilon$. Consider the smallest
  $i$ such that $\calO_i \cap \Tu_{\calO\cap\sfH_0} \neq \emptyset$ but
  $q_{i+1} \notin \Tu_{\calO\cap\sfH_0}$. Then there is a point
  $q' \in \Tu_{\calO\cap\sfH_0}$ in the unique geodesic of minimal length
  connecting $q_{i+1}$ with $\calO_i$ such that
  $d(q',\calO)\geq \frac{\varepsilon}{2}$. But this entails that
  \[
  \begin{split}
     d & (q_1,\calO_0) + \ldots + d(q_k,\calO_{k-1})  = \\
      &=  d(q_1,\calO_0) + \ldots + d(q_i,\calO_{i-1}) +
    d(q',\calO_i) + d(q_{i+1}, q') + \ldots + d(q_k,\calO_{k-1}) \geq
    \frac{\varepsilon}{2} ,
  \end{split}
  \]
  hence that $\overline{d}(\calO,\tilde\calO) \geq \frac{\varepsilon}{2}>0$.
  This finishes Step 3 and proves in particular that $\overline{d}$ is
  a metric on $X$ indeed.
  \\[1mm]
  {\bf Step 4.}
  Next we show that the canonical projection $\orbproj : \sfH_0 \rightarrow X$ is a
  submetry that means that
  $\orbproj \big( B_r (p) \big) = B_r \big(\orbproj (p)\big)$ for all $p\in \sfH_0$
  and $r>0$. Let $q \in  B_r (p)$, and denote by $\calO$ the orbit through
  $p$ and by $\calO_q$ the orbit through $q$. In other words let $\calO = \orbproj (p)$
  and $\calO_q = \orbproj (q)$. Let us prove that $\overline{d}(\calO,\calO_q) <r$.
  To this end choose a piecewise geodesic $\gamma:[0,\ell] \rightarrow \sfH_0$
  parameterized by arc length such that $\gamma (0)=p$ and $\gamma (\ell)=q$.
  Let $0=t_0 < t_1 < \cdots t_{k-1} <  t_k= \ell$
  be a partition such that $\gamma_{|[t_{i-1},t_i]}$ is a geodesic for $i=1,\cdots,k$.
  Let $q_i:=\gamma (t_i)$ and $\calO_i$ be the orbit through $q_i$.
  Then $\calO_0,\ldots,\calO_k,q_1,\ldots,q_k$ is a connecting sequence of orbits and
  points in $\sfH_0$ between $\calO$ and $\calO_q$. Moreover,
  \[
    \overline{d}(\calO,\calO_q) \leq d(\calO_0, q_1 ) + \ldots + d(\calO_{k-1},q_k)
    \leq (t_1-t_0) + \ldots + (t_k - t_{k-1}) =
    \ell < r,
  \]
  where $t_i - t_{i-1} $ is the length of the path  $\gamma_{[t-{i-1},t_i]}$.
  This proves that $\orbproj (q) \in B_r \big(\orbproj(p)\big)$.

  Now let $\tilde\calO$ be an orbit such that $\overline{d}(\calO,\tilde\calO) <r$.
  Then there exists a connecting sequence $\calO_0,\ldots,\calO_k,q_1,\ldots,q_k$
  of orbits and points in $\sfH_0$ between $\calO$ and $\tilde\calO$ such that
  \[
    d(\calO_0, q_1 ) + \ldots + d(\calO_{k-1},q_k) < r.
  \]
  Now put $p_0:= p$, and choose $g_1\in \sfH_0$ such that
  $p_0= g_1 \pi_{\calO_0} (q_1)$. Put $p_1 := g_1 q_1$. Note that $p_1$ is well-defined
  by Thm.~\ref{Thm:SliceThm2} since $g_1\in \sfH_{|\calO_0\cap\sfH_0}$.
  By construction one obtains $p_1\in \Tu_{\calO_0\cap\sfH_0}$ and
  $d(\calO_0,q_1) =d(p_0,p_1)$. Assume that we have constructed a sequence
  $p_0,\ldots,p_l\in \sfH_0$, $l<k$, such that $p_{i+1}\in \Tu_{\calO_i\cap\sfH_0}$ and
  $d(\calO_i,q_{i+1}) =d(p_i,p_{i+1})$ for $i=0,\ldots , l-1$.
  Choose $g_{l+1}\in \sfH_1$ such that $p_l = g_{l+1}\pi_{\calO_l} (q_{l+1}) $.
  Put $p_{l+1} := g_{l+1} q_{l+1}$. Then $p_{l+1}\in \Tu_{\calO_l\cap\sfH_0}$ and
  $d(\calO_l,q_{l+1}) =d(p_l,p_{l+1})$. Thus we have constructed inductively
  $p_0,\ldots , p_k$ such that
  $d(\calO_i,q_{i+1}) =d(p_i,p_{i+1})$ for $i=0,\ldots , k-1$.
  Let $\gamma_i : [0,\ell_i] \rightarrow \sfH_0$ be the geodesic
  parameterized by arc length connecting $p_i$ with $p_{i+1}$, and let
  $\gamma :[0,\ell] \rightarrow \sfH_0$ with $\ell:=\ell_0 +\ldots +\ell_{k-1}$
  the composition of paths
  $\gamma_{k-1} \star \ldots \star \gamma_0$. Then $\ell$ is the length of $\gamma$
  and
  \[
    \ell:= d(p_0,p_1) + \ldots + d(p_{k-1},p_k) =
    d(\calO_0, q_1 ) + \ldots + d(\calO_{k-1},q_k) < r.
  \]
  This implies that $p_k \in B_r (p)$ and  finishes the proof that $\pi$ is
  a submetry. Finally in Step 4 observe that $\pi$ being a submetry implies
  that the topology induced by the metric $\overline{d}$ on $X$ coincides
  with the quotient topology of $\pi$.
  \\[1mm]
  {\bf Step 5.}
  Since the geodesic distance on $\sfH_0$ is a length space metric
  and $\orbproj : \sfH_0 \rightarrow X$  a submetry, it follows by
  \cite[Prop.~1]{BurBurIvaCMG} that the metric $\overline{d}$ on $X$ is a
  length space metric as well.
  Since a length space metric is determined by its values on an arbitrarily small
  neighborhood of the diagonal, it follows by Eq.~\eqref{Eq:QuotMetProp}
  that $\overline{d}$ is independent of the particular choice of $\sfH_0$
  or the tubular neighborhoods $\Tu_\calO$, and is uniquely determined
  Eq.~\eqref{Eq:QuotMetProp}.
 \\[1mm]
  {\bf Step 6.}
In this last step, we show that the metric $\bar{d}$ constructed using a choice of $\mathsf{H}\subset\sfG$ equals the one induced by $\sfG$ by means of formula \eqref{semi-metric-G}, thereby proving that this indeed defines a metric, not just a semi-metric. Denote by $\bar{d}_\sfH$ and $\bar{d}_\sfG$ the two (semi-)metrics constructed using $\sfH$ and $\sfG$. Clearly we have that $\bar{d}_\sfG\leq \bar{d}_\sfH$. Suppose now that $\bar{d}_\sfG< \bar{d}_\sfH$. This means that there exists a connecting sequence $(\calO_i,q_i),~i=0,\ldots,k$ in $\sfG$ such that
\[
\sum_{i=1}^k d(q_i,\calO_{i-1})<\bar{d}_\sfH(\calO_0,\calO_k),
\]
and $q_i\not\in \sfH_0$ for some $i$. But by choosing balls around
those points $q_i$, one constructs a different $\sfH'$, also Morita
equivalent to $\sfG$, such that $q_i\in\sfH'_0$ for all $i=1,\ldots,
k$. But then we have
\[
\bar{d}_{\sfH'}(\calO_0,\calO_k)\leq \sum_{i=1}^k d(q_i,\calO_{i-1}).
\]
It follows from step 5 that $\bar{d}_\sfH=\bar{d}_{\sfH'}$, contradicting the previous inequality, and therefore we see that $\bar{d}_\sfG=\bar{d}_\sfH$.

 \vspace{3mm}

  We have now the crucial ingredients for the

\proofthm{Thm:LocMetSrucOrbitspace}
 (1) to (3) have been proven in Step 1 to Step 5.
\begin{enumerate}[({ad} 1)]
\setcounter{enumi}{3}
  \item
    By \cite[Prop.~1]{BurBurIvaCMG} again, $X$ inherits from
    $(\sfG_0,d)$ the property of being complete.
    Hence if $(\sfG_0,d)$ is complete, then $X$ is complete as well, and
    by the Theorem of Hopf--Rinow--Cohn--Vossen
    (cf.~\cite[Thm.~2.5.28]{BurBurIvaCMG}) every closed bounded ball in $X$ is
    compact.
  \item
    According to \cite[p.~16]{BurGroPerADASCBB} the property of being an
    Alexandrov space of Hausdorff dimension $\leq d$ is inherited by its quotient space
    under a submetry.
\end{enumerate}
\proofend
\section{Triangulation of  the orbit space of a proper Lie groupoid}
\label{sec:triangularization}
In this section we will show that the orbit space of a proper Lie groupoid
has a triangulation subordinate to the stratification by orbit types and then
will derive some consequences from this. 

Recall that a \emph{simplicial complex} is a collection $\calK$ 
of (closed) simplices in some $\R^n$ such that every face of a simplex of $\calK$ is
in $\calK$ and such that the intersection of any two simplices of $\calK$ is a
face of each of them, see \cite[\S 2]{MunkresBook}. 
A \emph{triangulation} of a topological space $X$ then consists of a
simplicial complex $\calK$ together with a homeomorphism 
$\kappa : |\calK| \rightarrow X$, where 
$|\calK| := \bigcup\limits_{\sigma \in \calK} \sigma$ is the underlying space 
of the simplicial complex. 
In case the topological space $X$ carries a stratification $\calS$, we 
say that a triangulation  $(\calK,\kappa)$ of $X$ is \emph{subordinate} to 
$\calS$, if for each stratum $S$ the preimage $\kappa^{-1} (\overline{S})$ is the
underlying space of a simplicial subcomplex of $\calK$, cf.~\cite[Sec.~3]{gor}. 

\begin{theorem}[{\cite[5.~Proposition]{gor}, \cite[3.7.~Corollary]{verona}}]
  Let $(X,\calS)$  be a controllable stratified space. Then there exists a 
  triangulation $(\calK,\kappa)$ of $X$ which is compatible with $\calS$. 
\end{theorem}

Since by Cor.~\ref{cor:stratification} the orbit space of a proper Lie groupoid is
a controllable stratified space, the theorem immediately implies
the following result.

\begin{corollary}
\label{cor:triangulation}
  Let $X$ be the orbit space of a proper Lie groupoid $\sfG$. Then there exists 
  a triangulation of $X$ which is compatible  with the stratification of $X$
  by orbit types.
\end{corollary}
 
Recall now that by  a \emph{good covering} of a topological space $X$ 
one understands an open covering $\calV$  such that for every finite family 
$V_1,\cdots,V_k$ of elements of $\calV$ the intersection 
$V_1 \cap \cdots \cap V_k$ is either empty or contractible. It appears to be 
folklore that every triangulable topological space has a good open covering. 
Since we could not find a full proof in the literature, we present here an outline 
of the argument based on a remark by \cite[p.~190]{BoTu}. 

So assume to be given a  topological space $X$ with a triangulation 
$(\calK,\kappa)$. 
Denote for every simplex $\sigma \in \calK$ by $\clstar{\sigma}$ the 
\emph{closed star} of $\sigma$, which is defined by
\[
 \clstar{\sigma} := \bigcup_{\tau \in \calK, \: \sigma \subseteq \tau} \kappa (\tau) ,
\]
and by $\opstar{\sigma}$ the \emph{open star} of $\sigma$, 
i.e.~the interior of $\clstar{\sigma}$. Let $\calK_0$ be the set of all vertices or
in other words $0$-simplices of $\calK$. The set of open stars $\opstar{v}$, $v\in \calK_0$
then forms an open covering of $X$. Moreover, if $v_1,\ldots,v_k \calK_0$ are finitely many 
pairwise distinct vertices such that
\[
   \opstar{v_1} \cap \ldots \cap \opstar{v_k} \neq \emptyset,
\]
then $\opstar{v_1} \cap \ldots \cap \opstar{v_k}$ is contractible. Let us briefly provide
a sketch of this. Let $\sigma$ be the $(k-1)$-simplex from $\calK$ having vertices 
$v_1,\ldots,v_k$. Let $v$ be the barycenter of $\sigma$. Since the image of the 
interior of $\sigma$ under $\kappa$ lies in 
$\opstar{v_1} \cap \ldots \cap \opstar{v_k}$, the point $\kappa (v) $ is 
an element of $\opstar{v_1} \cap \ldots \cap \opstar{v_k}$. If now $x$ is 
another point in $\opstar{v_1} \cap \ldots \cap \opstar{v_k}$, then $\kappa^{-1}(x)$
lies in a simplex $\tau \in \calK$ which has $v_1,\ldots ,v_k$ among its vertices.
Since $\tau$ is convex, the straight line 
$\ell_{\kappa^{-1}(x),v}$ between $\kappa^{-1} (x)$ and $v$ runs in $\tau$, hence 
in $|\calK|$. This implies that 
$\kappa (\ell_{\kappa^{-1}(x),v}) \subset \opstar{v_1} \cap \ldots \cap \opstar{v_k}$, hence
$\opstar{v_1} \cap \ldots \cap \opstar{v_k}$ is contractible to $\kappa (v)$. 
In other words this means that the open covering $\big( \opstar{v} \big)_{v \in \calK_0}$
is a good covering of $X$.

If the topological space $X$ with triangulation $(\calK,\kappa)$ is even locally compact, 
the simplicial complex $\calK$ has to be locally finite (cf.~\cite[Lemma 2.6]{MunkresBook}).
Given an open covering $\calU$ of $X$, one can construct by iterated barycentric subdivison 
of $\calK$ a simplicial complex $\calL$ such that $|\calL| = |\calK|$ and such that 
for each simplex $\sigma \in \calL$ there exists an open set $U_\sigma\in \calU$ with
$\kappa (\sigma) \subset U_\sigma$. 
After further subdivison one can even achieve that for each vertex $v\in \calL_0$ there 
exists a $U_v \in \calU$ with $\opstar{v}\subset \calU$. By Cor.~\ref{cor:triangulation}
and since orbit spaces of proper Lie groupoids are locally compact, one obtains

\begin{proposition}
\label{prop:goodcov}
  Let $X$ be the orbit space of a proper Lie groupoid $\sfG$, and $\calU$ an open covering
  of $X$. Then there exists a locally finite good covering of $X$ which is subordinate to 
  $\calU$. 
\end{proposition}

\begin{remark}
 The existence of good coverings on the quotient space of a singular riemannian 
 foliation on a compact manifold has been proven in Prop.~1 in \cite{Wolak} under 
 the (implicitly made) assumption that for each closure of a leave $L$ of the 
 foliation there is an $\varepsilon >0$ such that that any two closures of leaves 
 in $B_\varepsilon (\overline{L})$ can be joint by an orthogonal
 geodesic  contained in $B_\varepsilon (\overline{L})$.
 It is not clear to us whether every singular riemannian foliation satisfies 
 this condition. Prop.~\ref{prop:goodcov} above is not dependent on such 
 a condition and therefore appears to be a new result. 
\end{remark}
\section{A de Rham theorem}
\label{sec:derham}
In this section we shall prove a de Rham theorem for quotients of proper Lie groupoids
in terms of \textit{basic differential forms}. Let $\sfG$ be Lie groupoid, for the moment 
not necessarily proper. We denote by $(A,\rho)$ its Lie algebroid, where 
$\rho:A\to T\sfG_0$ is the anchor map. 
\begin{definition}
\label{basic}
A differential form $\alpha\in\Omega^\bullet(\sfG_0)$ is said to be 
\textit{basic}, if it satisfies the following two conditions:
\begin{enumerate}[(B1)]
\item $\iota_{\rho(X)}\alpha=0$, for all $X\in\Gamma^\infty(A)$.
\item $\alpha$ is (right) $\sfG$-invariant.
\end{enumerate}
\end{definition}
Observe that since the image of the anchor point-wise spans $T\calO$, the tangent spaces to 
the orbits of $\sfG$, $\alpha$ defines a section of the dual of $N\calO$ by $(B1)$, and 
therefore $(B2)$ makes sense. This definition is of course a straightforward 
generalization of the well-known concept of a basic form with respect to an action of 
a Lie group. In terms of the Lie algebroid, condition $(B2)$ above implies that 
$L_{\rho(X)}\alpha=0$ for all $X\in\Gamma^\infty(A)$, and when $\sfG$ is 
source-connected, this is in fact equivalent to $(B2)$.
Using Cartan's formula, one observes that the de Rham differential on 
$\Omega^\bullet(\sfG_0)$ preserves the basic forms. We write 
$\left(\Omega^\bullet_\textup{basic}(\sfG), d\right)$ for the complex of basic forms defined
in this way. Its cohomology is called the \textit{basic cohomology} of $\sfG$, denoted 
$H^\bullet_\textup{basic}(\sfG,\R)$.
\begin{remark}
This definition of basic cohomology generalizes the following two well-known cases:
\begin{itemize}
\item[$i)$] When $\sfG$ is the action groupoid of a Lie group action on a manifold, the definition above
yields the basic cohomology of a quotient space, as in \cite{koszul}.
\item[$ii)$] A Lie groupoid $\sfG$ is a foliation groupoid when the anchor of its Lie algebroid is injective \cite{mm}. In this case one finds the basic cohomology of the leaf space of the associated foliation, defined in \cite{reinhart}.\end{itemize} 
\end{remark}
\begin{proposition}
\label{inv-we}
A weak equivalence $f:\sfG\to\sfH$ induces an isomorphism 
\[
f^*:\left(\Omega^\bullet_{\rm {basic}}(\sfH),d\right)\stackrel{\cong}{\longrightarrow}\left(\Omega^\bullet_{\rm {basic}}(\sfG),d\right),
\]
and therefore an isomorphism on basic cohomology.
\end{proposition}
\begin{proof}
Clearly, any morphism of Lie groupoids $f:\sfG\to\sfH$ induces by pullback a morphism $f^*:\Omega^\bullet_{\rm {basic}}(\sfH)\to\Omega^\bullet_{\rm {basic}}(\sfG)$ commuting with the de Rham differential.
To see that this is an isomorphism when $f$ is a weak equivalence, consider the space $M:=\sfH_1\times_{\sfH_0}\sfG_0$ appearing in condition (ES) of a weak equivalence as stated above.
There are two obvious maps $pr_2:M\to\sfG_0$ and $M\to\sfH_0$, and these are the so-called moment maps for actions of $\sfG$ and $\sfH$ on $M$. These actions are in fact biprincipal, meaning that $\sfG_1\backslash M\cong\sfG_0$ and $M\slash\sfH_1\cong\sfH_0$. (Such objects are also called \textit{Morita bibundles}.) Now if we have a basic form 
$\beta\in\Omega^\bullet_\textup{basic}(\sfG)$, we can consider its pullback $pr_2^*\beta$ to $M$. Since $f$ induces a morphism between the Lie algebroids of $\sfG$ and $\sfH$, one can easily check that $pr_2^*\beta$ is basic with respect to the action of $\sfH$. (The notion of a basic differential form with respect to an action of a Lie groupoid is a straightforward generalization of Definition \ref{basic} above.) Therefore, since the action of $\sfH$ is principal, this form descends to $\sfG_0$, where it is, by construction, $\sfH$-basic.
These two maps are each others inverse, proving the claim.
\end{proof}
\begin{corollary}
Basic cohomology of Lie groupoids is Morita invariant.
\end{corollary}
We now assume that $\sfG$ is proper, so that the quotient $X$ is stratified with a smooth structure.
Over $X$, there is an obvious sheafification of the basic differential forms, resulting in the sheaf
$\Omega^\bullet_\textup{basic}$ that assigns to each open set $U\subset X$ the space  of basic differential 
forms $\Omega^\bullet_\textup{basic}(\sfG_{|\pi^{-1}(U)})$, where $\pi:\sfG_0\to X$ is the canonical projection.
Clearly, $\Omega^\bullet_\textup{basic}(X)$ is the space of basic differential forms on $\sfG_0$ as discussed above. Remark that in degree zero, $\Omega^\bullet_\textup{basic}(X)=C^\infty(X)$.
\begin{lemma}[Poincar\'e Lemma]
Let $\alpha$ be a closed basic differential form on $\sfG_0$. Then each $\calO\in X$ 
has an open neighborhood $U$ together with $\beta\in\Omega^\bullet_\textup{basic}(U)$ 
such that $\alpha=d\beta$.
\end{lemma}
\begin{proof}
By Zung's theorem, combined with Proposition \ref{loc-morita} and the Morita invariance 
of basic differential forms, cf.\ Proposition \ref{inv-we} above, it suffices to prove 
the statement for the quotient of an action of a compact Lie group. But this case is 
well-known, see e.g.~\cite[\S 5.3]{pf:book}
\end{proof}
\begin{proposition}
The sheaf complex
\[
0\to\underline{\R}\longrightarrow \calC^\infty_X\stackrel{d}{\longrightarrow} 
\Omega^1_\textup{basic}\stackrel{d}{\longrightarrow} 
\Omega^2_\textup{basic}\stackrel{d}{\longrightarrow} \ldots
\]
forms a fine resolution of $\underline{\R}$, the sheaf of locally constant 
functions on $X$ with values in $\R$.
\end{proposition}
\begin{proof}
The previous Poincar\'e Lemma shows that the sequence is exact. Clearly, each $\Omega^\bullet_\textup{basic}$ is a sheaf of $C^\infty_X$-modules, and therefore fine.
\end{proof}
 Since by remark \ref{Rem:LocCon} the space $X$ is locally contractible and 
 locally compact, singular cohomology of $X$ with values in $\R$ coincides 
 with the sheaf cohomology of $X$ with values in $\underline{\R}$. 
 The proposition therefore entails 
\begin{corollary}[de Rham theorem]
There is a natural isomorphism
\[
 H^\bullet_\textup{sing}(X,\R)\cong H^\bullet(X,\underline{\R})\cong
 H^\bullet_\textup{basic}(\sfG,\R).
\]
\end{corollary}
Since the left hand side in the Corollary  is the singular cohomology of a 
triangulable locally compact topological space by 
Cor.~\ref{cor:triangulation}, it naturally coincides with \v{C}ech cohomology.
Prop.~\ref{prop:goodcov} then implies our final result.
\begin{corollary}
  For a proper Lie groupoid $\sfG$ with compact quotient $X$, basic 
  cohomology $H^\bullet_\textup{basic}(\sfG,\R)$ is finite dimensional.
\end{corollary}
\appendix
\section{Ehresmann's Theorem for manifolds with boundary}
A classical theorem by Ehresmann \cite{EhrSub} 
says that a proper surjective submersion
between smooth manifolds is locally  trivial. In this appendix, we present  
a generalized version of Ehresmann's Theorem, where the domain of  the 
submersion under consideration is allowed to be a manifold with boundary.
To our knowledge, this result, which we name Ehresmann's Theorem
for manifolds with boundary, has been proved first in the diploma thesis 
of Gutfleisch \cite{GutSRQA}. Since this work might not be accessible to the 
reader, we present here a proof. Note that our proof is different to
the one by Gutfleisch, and uses the local triviality theorem by Mather  
\cite{Mat:NTS}.

\begin{theorem}[Ehresmann's Theorem for manifolds with boundary]
\label{ThmEhresmannMfdBdry}
  Let $M$ be a smooth manifold without boundary, $N$ a smooth manifold with 
  boundary, and $f: N \rightarrow M$  a proper surjective submersion such 
  that the restriction $f_{| \partial N}$ is a submersion as well. 
  Then $f$ is locally trivial. Moreover, each fiber $F_p := f^{-1} (p)$, 
  $p\in M$ is a compact manifold with boundary. Its boundary is given by 
  $\partial F_p = \partial N \cap F_p$.  
\end{theorem}
\begin{proof}
  First observe that the manifold with boundary $N$ carries 
  a canonical stratification given by the decomposition into the interior 
  $N^\circ$ and the boundary $\partial N$. Since   $f_{| \partial N}$ is a 
  submersion by assumption, and $M$ a manifold without boundary, there exists a nowhere vanishing  
  ``inward pointing'' smooth vector field $V:U \rightarrow TN$ on an open 
  neighborhood $U$ of $\partial N$ in $N$ such that $Tf \circ V =0$ and 
  $T_qN = T_q{\partial N}  \oplus \R \cdot V(q)$ for all $q \in \partial N$.
  %
  %
  Then there exists a continuous map 
  $\varepsilon :\partial N \rightarrow \R^{>0}$ and for each $q\in \partial N$ 
  a uniquely determined smooth  integral curve 
  $\gamma_q  : [0,\varepsilon (q)) \rightarrow U$ of $V$ such  $\gamma_q (0) =q$ 
  for all $t\in [0,\varepsilon (q))$.
  Put 
  $[0,\varepsilon ) := \{ (q,t) \in \partial N \times \R^{\geq 0}\mid   
   t < \varepsilon (q) \}$
  and 
  \[
    \sfT_{\partial N} := \{ \gamma_q (t) \in U \mid  q \in \partial N \: \&  \:
    t\in [0,\varepsilon (q)) \}  .
  \]
  Then 
  \[
   \gamma : [0,\varepsilon )  \rightarrow \sfT_{\partial N}, \: (q,t) \mapsto
   \gamma_q(t)
  \] 
  is a diffeomorphism, and we can define 
  $\pi_{\partial N} : \sfT_{\partial N} \rightarrow \partial N$ 
  and $\varrho_{\partial N} : \sfT_{\partial N} \rightarrow  \R^{\geq 0}$ 
  as the uniquely determined smooth maps such that 
  $ (\pi_{\partial N} , \varrho_{\partial N}) $ is the inverse of $\gamma$. 
  By construction, $\partial N = \varrho_{\partial N}^{-1} (0)$, and
  $(\sfT_{\partial N} , \pi_{\partial N} , \varrho_{\partial N}) $ is a smooth tube 
  of $\partial N$ in the stratified space $N$. 
  Moreover, by the choice of $V$, one has that
  \[
    f ( \pi_{\partial N} (q)) = f (q) \:\text{ for all $q \in \sfT_{\partial N} $ }.
  \]
  Hence, $f:N \rightarrow M$ is a controlled submersion in the sense of 
  Mather \cite[\S 9, p.~46]{Mat:NTS}. Since by Mather \cite[Cor.~10.2]{Mat:NTS}
  every proper surjective controlled submersion from a  
  stratified space with smooth control data to a smooth manifold is 
  locally trivial, the submersion $f$ has to be locally trivial. 
  The claim on the fibers is obvious by the assumptions on $f$.
\end{proof}
\bibliographystyle{alpha}

\begin{thebibliography}{}
\bibitem[{\sc AbCr}]{ac} {\sc Abad, C.} and {\sc M. Crainic}:
   \textit{Representations up to homotopy and Bott's spectral sequence for Lie
   groupoids}. \texttt{arXiv:0911.2859} (2009).

\bibitem[{\sc Ber}]{BerSSFNC} {\sc Berstovskii, V.N.}:
   \textit{Submetries of space-forms of nonnegative curvature}.
   (Russian) Sibirsk. Mat. Zh. \textbf{28}, no. 4, 44--56 (1987).

\bibitem[{\sc BoTu}]{BoTu} {\sc Bott, R.}, and {\sc L.W.~Tu}:
   \textit{Differential Forms in Algebraic Topology}, Graduate Texts in Mathematics
   \textbf{82}, Springer Verlag, New York, Berlin (1982).

\bibitem[{\sc BuBuIv}]{BurBurIvaCMG}
   {\sc Burago, Burago} and {\sc Ivanov}:
   \textit{A Course in Metric Geometry}.
   AMS Graduate Studies in mathematics Vol.~\textbf{33} (2001).

\bibitem[{\sc BuGrPe}]{BurGroPerADASCBB}
   {\sc Burago, Y., M.~Gromov} and {\sc G.~Perelmann}:
   \textit{A.D.~Alexandrov spaces with curvature bounded below}.
   Russian Math.~Surveys \textbf{42}(2), 1--58 (1992).


\bibitem[{\sc Cra}]{Cra:pc} {\sc Crainic, M.}: personal communication (2011).  

\bibitem[{\sc CrFe}]{cf} {\sc Crainic, M.} and {\sc R. Fernandes}:
 \textit{Integrability of Lie brackets.} Ann. of Math. (2)
 \textbf{157} (2003).

\bibitem[{\sc CrSt}]{cs} {\sc Crainic, M.} and {\sc I. Struchiner}:
  \textit{On the linearization theorems
  for proper Lie groupoids}, \texttt{arXiv:1103.5245} (2011).

\bibitem[{\sc DeFe}]{hofe} {\sc del Hoyo, M.} and {\sc  R. Fernandes}: personal communication (2013).

\bibitem[{\sc Dra}]{dr:thesis}
   {\sc Dragulete, O.}: {\it Some applications of symmetries in
differential geometry and dynamical systems}, \'Ecole Polytechnique
F\'ed\'erale de Lausanne, PhD Thesis, (2007).

\bibitem[{\sc DuZu}]{DufZunPSNF}
  {\sc Dufour, J.-P.} and {\sc N.~T.~Zung}: \textit{Poisson structures and their
  normal forms}, Progress in Mathematics, vol.~\textbf{242}, 
  Birkh\"auser Verlag, Basel, (2005).

\bibitem[{\sc DuKo}]{dk:book}
   {\sc Duistermaat, J.J.} and {\sc Kolk, J.A.C.}: \textit{Lie group}, 
   Springer-Verlag, Heidelberg, (2000). 

\bibitem[{\sc Ehr}]{EhrSub} 
   {\sc Ehresmann, C.}:
   \textit{Sur les espaces fibr\'es diff\'erentiables}, 
   C.~R.~Acad.~Sci.~Paris \textbf{224}, 1611--1612 (1947).

\bibitem[{\sc Fer}]{fernandes} {\sc Fernandes, R.}: 
  \textit{Lie algebroids, holonomy and characteristic classes}. 
  Adv. Math.  \textbf{170}, no. 1, (2002).

\bibitem[{\sc GGHR}]{gghr} E. Gallego, L. Gualandri, G. Hector, and A. Reventos. Groupo\"ides
riemanniens. Publ. Mat. \textbf{33} (1989),
no. 3, 417--422.

\bibitem[{\sc Gli}]{gl}  {\sc Glickenstein, D.:} 
  \textit{Riemannian groupoids and solitons for three-dimensional homogeneous Ricci and cross-curvature flows}. Int. Math. Res. Not. IMRN 2008, no. 12

\bibitem[{\sc GoSa}]{GonSalDS}
 J.A.N.~Gonz\'alez and J.B.S.~de Salas, \emph{$\calC^\infty$-Differentiable Spaces},
 Lecture Notes in Mathematics \textbf{824}, Springer, Berlin, 2003.

\bibitem[{\sc Gor}]{gor} {\sc Goresky, R.M.:} \textit{Triangulation of Stratified Sets}.
Proc. Amer. Math. Soc. \textbf{72}, Nr.~1, 193--200 (1980).

\bibitem[{\sc Gro}]{GroMSRNRS}
  {\sc Gromov, M.}:
  \emph{Metric Structures for Riemannian and Non-Riemannian
  Spaces}, Progress in Mathematics Vol.~\textbf{152}, Birkh\"auser (1998).

\bibitem[{\sc Gut}]{GutSRQA}
  {\sc Gutfleisch, M.}: 
  \emph{Stratifizierte R\"aume und Qutienten in der Analysis}, 
  Diploma Thesis, University of Munich, February 1995. 
  
\bibitem[{\sc Kan}]{Kan:}
  {\sc Kankaanrinta, M.:}
  \emph{Equivariant collaring, tubular neighborhood and gluing theorems for
  proper Lie group actions}.
  Algebraic \& Geometric Topology \textbf{7}, 1--27 (2007).

\bibitem[{\sc Kli}]{KliRG}
   {\sc Klingenberg, W.}: {\it Riemannian Geometry}
   Walter de Gruyter Publisher, (1982).

\bibitem[{\sc Kos}]{koszul} Koszul, J. L.
Sur certains groupes de transformations de Lie. In: \textit{G\'eom\'etrie diff\'erentielle}. Colloques Internationaux du CNRS, Strasbourg, 1953, pp. 137?141. CNRS, Paris, 1953.

\bibitem[{\sc Mat70}]{Mat:NTS}
  {\sc Mather, J.N.}: {\it Notes on topological stability}, Mimeographed Lecture Notes,
  Harvard, (1970).

\bibitem[{\sc Mat73}]{ma:stratification}
   {\sc Mather, J.N.}: {\it Stratifications and mappings}, Dynamical systems
(Proc. Sympos., Univ. Bahia, Salvador, 1971), pp. 195--232. Academic
Press, New York, (1973).

\bibitem[{\sc MoMr}]{mm}
  {\sc Moerdijk, I.} and {\sc Mr\v{c}un, J.}:
  \textit{Introduction to foliations and Lie groupoids.}
  Cambridge Studies in Advanced Mathematics, \textbf{91}. Cambridge University
  Press, Cambridge, (2003).

\bibitem[{\sc Mol1}]{MolRF}
  {\sc Molino, P.}: \emph{Riemannian foliations}, Progress in Mathematics, vol.~73,
  Birkh\"auser Boston Inc., Boston, MA, 1988, Translated from the French by
  Grant Cairns, With appendices by Cairns, Y. Carri{\`e}re, {\'E}. Ghys, E.
  Salem and V. Sergiescu.

\bibitem[{\sc Mol2}]{Molorb}
  {\sc Molino, P.}: {\it Orbit-like foliations}, 
  Geometric study of foliations (Tokyo, 1993),
  World Sci. Publ., River Edge, NJ (1994).

\bibitem[{\sc Mun}]{MunkresBook}
  {\sc Munkres, J.R.:} \textit{Elements of Algebraic Topology},
  Perseus Publishing, Cambridge, MA (1984). 

\bibitem[{\sc Pfl01a}]{pf:orbit}
   {\sc Pflaum, M.}: {\it Analytic and geometric study of stratified
spaces}, Lecture Notes in Mathematics, \textbf{1768},
Springer-Verlag, Berlin, (2001).

\bibitem[{\sc Pfl01b}]{pf:book}
   {\sc Pflaum, M.}: {\it Smooth structures on stratified spaces}, 
   Quantization of singular symplectic quotients,
   231--258, Progr. Math., \textbf{198}, Birkh\"user, Basel, (2001).

\bibitem[{\sc Rei}]{reinhart} Reinhart, B. L.
Harmonic integrals on foliated manifolds. 
\textit{Amer. J. Math.} \textbf{81} 1959 529?536.

\bibitem[{\sc Schw72}]{Schwa:thesis}
    {\sc Schwarz, G.W.}:
    {\it Generalized Orbit Spaces}, PhD Thesis (1972).

\bibitem[{\sc Schw75}]{Schwa:SFIACLG}
    {\sc Schwarz, G.W.}:
    {\it Smooth functions invariant under the action of a compact
    {Lie} group}, Topology \textbf{14} (1975), 63--68.

\bibitem[{\sc Ste}]{SteASOFS}
  {\sc Stefan, P.}:
  \emph{Accessible sets, orbits, and foliations with singularities},
  Proc. London Math. Soc. (3) \textbf{29} (1974), 699--713.

\bibitem[{\sc Sus}]{SusOFVFID}
  {\sc Sussmann, H.J.}:
  \emph{Orbits of families of vector fields and
  integrability of distributions}, Trans. Amer. Math. Soc. \textbf{180} (1973),
  171--188.


\bibitem[{\sc Tu}]{tu} {\sc Tu, J.L.:} \textit{La conjecture de Novikov pour les
feuilletages hyperboliques.} K-theory \textbf{16} (1999) p. 129--184.

\bibitem[{\sc Ver}]{verona}
  {\sc Verona, A.:} \textit{Triangulation of Stratified Fibre Bundles}.
  manuscripta math.~\textbf{30}, 425--445 (1980).

\bibitem[{\sc Wei99}]{w:linear1}
   {\sc Weinstein, A.}: {\it Linearization problems for Lie algebroids and Lie
groupoids}, Conf\'erence Mosh\'e Flato (1999), Vol. II (Dijon),
333--341, Math. Phys. Stud., \textbf{22}, Kluwer Acad. Publ.,
Dordrecht, (2000).

\bibitem[{\sc Wei02}]{w:linear2}
   {\sc Weinstein, A.}: {\it Linearization of regular proper groupoids}, J. Inst.
Math. Jussieu \textbf{1}, no. 3, 493--511, (2002).

\bibitem[{\sc Wol}]{Wolak}
   {\sc Wolak, R.A.}: 
   \textit{Basic Cohomology for Singular Riemannian Foliations},
   Mh. Math. \textbf{128}, 159--163 (1999).

\bibitem[{\sc Zun}]{Zun}
   {\sc Zung, N.T.}: {\it Proper groupoids and momentum maps: linearization,
affinity, and convexity}, Ann. Sci. \'ecole Norm. Sup. (4)
\textbf{39}, no. 5, 841--869, (2006).

\end{thebibliography}

\end{document}